\documentclass[10pt,a4paper]{amsart}
\usepackage[latin1]{inputenc}
\usepackage{mathrsfs}
\usepackage{dsfont}
\usepackage{amsmath,amssymb}
\newtheorem{theorem}{Theorem}

\newtheorem{remark}{Remark}
\newtheorem{remarks}[remark]{Remarks}
\newtheorem{lemma}{Lemma}
\newtheorem{corollary}{Corollary}
\newtheorem{definition}{Definition}
\newcommand{\ii}{\infty}
\newcommand\R{{\ensuremath {\mathbb R} }}
\newcommand\C{{\ensuremath {\mathbb C} }}

\renewcommand\phi{\varphi}
\newcommand{\gH}{\mathfrak{H}}
\newcommand{\gS}{\mathfrak{S}}

\newcommand{\cP}{\mathcal{P}}

\newcommand{\cX}{\mathcal{X}}
\newcommand{\cB}{\mathcal{B}}
\newcommand{\cK}{\mathcal{K}}
\newcommand{\cE}{\mathcal{E}}
\newcommand{\cF}{\mathcal{F}}
\newcommand{\cN}{\mathcal{N}}

\newcommand{\cH}{\mathcal{H}}

\renewcommand{\epsilon}{\varepsilon}
\newcommand\pscal[1]{{\ensuremath{\left\langle #1 \right\rangle}}}
\newcommand{\norm}[1]{ \left| \! \left| #1 \right| \! \right| }
\newcommand{\tr}{{\rm Tr}}

\renewcommand{\geq}{\geqslant}
\renewcommand{\leq}{\leqslant}

\numberwithin{equation}{section}

\begin{document}

\title[Blowup for generalized Hartree-Fock equations]{On blowup for time-dependent \\ generalized Hartree-Fock equations}

\author[C. Hainzl]{Christian HAINZL}
 \address{Department of Mathematics, UAB, Birmingham, AL 35294-1170, USA.}
  \email{hainzl@math.uab.edu}

\author[E. Lenzmann]{Enno LENZMANN}
 \address{Department of Mathematical Sciences, University of Copenhagen, Universitetspark 5, 2100 Copenhagen \O, Denmark.}
  \email{lenzmann@math.ku.dk}

\author[M. Lewin]{Mathieu LEWIN}
 \address{CNRS \& Laboratoire de Mathématiques (CNRS UMR 8088), Universit\'e de Cergy-Pontoise, 2 Avenue Adolphe Chauvin, 95302 Cergy-Pontoise Cedex, France.}
  \email{Mathieu.Lewin@math.cnrs.fr}

\author[B. Schlein]{Benjamin SCHLEIN}
 \address{Centre for Mathematical Sciences, University of Cambridge, Wilberforce Road, Cambridge, CB3 0WB, UK.}
  \email{b.schlein@dpmms.cam.ac.uk }

\date{\today. \scriptsize~\copyright~2009 by the authors. This paper may be reproduced, in its entirety, for non-commercial~purposes.}

\begin{abstract}
We prove finite-time blowup for spherically symmetric and negative energy solutions of Hartree-Fock and Hartree-Fock-Bogoliubov type equations, which describe the evolution of attractive fermionic systems (e.\,g.~white dwarfs). Our main results are twofold: First, we extend the recent blowup result of [Hainzl and Schlein, \textit{Comm.\,Math.\,Phys.} \textbf{287} (2009), 705--714]  to Hartree-Fock equations with infinite rank solutions and a general class of Newtonian type interactions.  Second, we show the existence of finite-time blowup for spherically symmetric solutions of a Hartree-Fock-Bogoliubov model, where an angular momentum cutoff is introduced. We also explain the key difficulties encountered in the full Hartree-Fock-Bogoliubov theory.    
\end{abstract}

\maketitle


\section{Introduction}

In this paper, we study time-dependent generalized Hartree-Fock equations that arise as quantum fermionic models for the dynamics of relativistic stars (like white dwarfs and neutron stars).

Here, our particular interest is devoted to the existence of finite-time blowup solutions, which physically indicate the onset of ``gravitational collapse'' of a very massive star due to its own gravity. Indeed, such a catastrophic scenario was already predicted by the physicist S.~Chandrasekhar in 1930's, based on an intriguing and heuristic combination of special relativity and Newtonian gravity. However, it is fair to say that a rigorous understanding of the dynamics of the gravitational collapse of a star -- in terms of relativistic quantum mechanics and Newtonian gravity to start with -- is still at its beginning. In fact, the proposed mathematical models (such as the generalized Hartree-Fock equations considered below) are formulated as nonlinear dispersive evolution equations, which exhibit the essential feature {\em $L^2$-mass-criticality} and thus reflecting the physical fact of a so-called Chandrasekhar limit mass $$M_{*} \approx 1.4 M_{\rm Sun},$$ beyond which gravitational collapse of a cold star can occur. For a detailed discussion of Hartree-Fock type theory in astrophysics, we refer to \cite{FroLen-07, HaiSch-09} and references given there. See also \cite{LieThi-84,LieYau-87} for rigorous work on Chandrasekhar's theory in the time-independent setting.

Let us now briefly outline the main content of the present paper. As a relativistic quantum model of the dynamics of star, we consider time-dependent generalized Hartree-Fock equations with relativistic kinetic energy and an attractive two-body interaction of Newtonian type. More specifically, we study equations of {\em Hartree-Fock (HF)} and {\em Hartree-Fock-Bogoliubov (HFB)} type, and our main results on finite-time blowup of spherically symmetric solutions for these models will be stated in Sections \ref{sec:HF} and \ref{sec:HFB}, respectively. 

In fact, this paper aims at two directions: First, we consider the HF equations and we extend the recent blowup results of \cite{HaiSch-09} to solutions with infinite-rank density matrices and a broader class of two-body interactions. Second, we address the finite-time blowup for HFB type evolution equations, which also incorporate the physically important phenomenon of {\em Cooper pairing} for attractive fermion systems. However, it turns out that the analysis of HFB equations brings in new key difficulties, which were completely absent for HF. Most importantly, the nonconservation of higher moments of angular momenta in HFB theory makes an adaptation of the strategies used in \cite{FroLen-07,HaiSch-09} an insurmountable task at the present. To circumvent this difficulty, we introduce an HFB model with an {\em angular momentum cutoff} and prove a finite-time blowup result for this simplified model; see Section \ref{sec:HFB} for more details. In the long run, we do hope that our arguments developed in this paper will be of use to understand the finite-time blowup for the full-fledged HFB model without any cutoffs.

\subsection*{Outline of the paper}
This paper is organized as follows. In Sections \ref{sec:HF} and \ref{sec:HFB}, we first collect some preliminary facts and we then state our main results on finite-time blowup for spherically symmetric solutions of HF and HFB type models. Furthermore, we explain the key difficulties encountered for HFB theory. The proofs of the main theorem will be given in Section \ref{sec:proofs}, and Appendix A contains the proof of a technical lemma.

\subsection*{Conventions and notations.}
For the physically inclined reader, we remark that we employ units such that Planck's constant $\hbar$ and the speed of light $c$ both satisfy $\hbar = c =1$. As usual, the letter $C$ denotes a positive constant depending only on fixed quantities (such as initial data etc.) and $C$ is allowed to change from line to line.

\subsection*{Acknowledgments}
Part of this work was done during the authors' visit at the Erwin Schr\"odinger Institute for Mathematical Physics in Vienna, Austria, and the hospitality and support during this visit is gratefully
acknowledged. C.\,H.~is partially supported by NSF grant DMS-0800906. E.\,L.~acknowledges support from NSF grant DMS-070249 and a Steno fellowship from the Danish research council. M.\,L.~acknowledges support from the ANR project ACCQuaRel of the French ministry of research.

\section{Blowup in Hartree-Fock theory} \label{sec:HF}

In this section, we consider the Hartree-Fock (HF) model describing a system of pseudo-relativistic fermions with average particle number $\lambda > 0$ and attractive interactions of Newtonian-type (like classical gravity). With regard to physical applications, we may think of such system representing a {\em white dwarf star}, where effects due to special relativity play a significant role but gravity can still be treated in the Newtonian framework. Of course such a model lacks Poincar\'e covariance, and hence it is usually referred to  as {\em pseudo-relativistic}. Nonetheless, many essential effects (such as the existence of a {\em  Chandrasekhar limit mass} of white dwarfs) are already predicted -- even quantitatively within good approximation -- by such pseudo-relativistic models. Indeed, it should not go unmentioned here that S.~Chandrasekhar's original and acclaimed physical theory of gravitational collapse rests on a pseudo-relativistic model of a star in its semi-classical regime; see \cite{Chandra-31}.

 According to HF theory, the state of the system is described in terms of a \emph{density matrix}  $\gamma$ which is a self-adjoint operator acting on $L^2(\R^3;\C)$\footnote{For simplicity, we discard the spin of the particles throughout. But all our arguments can be easily generalized to particles having $q$ internal degrees of freedom, where $L^2(\R^3;\C)$ has to be replaced by $L^2(\R^3;\C^q)$ etc.
}, satisfying $0\leq\gamma\leq 1$ (reflecting the Pauli principle for fermions) and the normalization condition $\tr(\gamma)=\lambda$ fixing the (average) number of particles.

Let $V$ denote the interaction potential between the particles and, for notational convenience, let $\kappa > 0$ be a coupling constant controlling the strength of the interaction. Then the corresponding HF energy reads
\begin{equation}\label{eq:HF2} 
\cE_{\rm HF} (\gamma) = \tr(K  \gamma) + \frac{\kappa}{2} \iint V(x-y) \bigg(\rho_\gamma(x) \rho_\gamma(y) - |\gamma(x,y)|^2\bigg)\,dx\, dy \, . 
\end{equation} 
Here $\gamma(x,y)$ is the integral kernel of the trace-class operator $\gamma$ and $\rho_\gamma(x)=\gamma(x,x)$ is the associated density function. The pseudo-differential operator 
$$K=\sqrt{-\Delta + m^2}$$
describes the kinetic energy of a relativistic quantum particle with rest mass $m \geq 0$. 

In the present section, we are interested in the corresponding time-dependent HF equations, which can be formulated as the following initial-value problem:
\begin{equation}\label{eq:HF1} 
\boxed{\left\{\begin{array}{l}
\displaystyle i \frac{d}{dt} \gamma_t = \left[ H_{\gamma_t}  , \gamma_t \right], \\[1ex]
\displaystyle \gamma_{|t=0}=\gamma_0\in\cK_{\rm HF}.
\end{array}\right.}
\end{equation} 
Here $[A,B] = AB -BA$ denotes the commutator and
\begin{equation}
H_\gamma:=\sqrt{-\Delta+m^2} + \kappa \left( V * \rho_{\gamma} \right) - \kappa V(x-y)\gamma(x,y) 
\label{def:mean-field}
\end{equation}
is the so-called \emph{mean-field operator} which depends on $\gamma$ and acts on the one-body space $L^2(\R^3;\C)$. Here and henceforth, the symbol $*$ stands for convolution of functions on $\R^3$. For expositional convenience, we use a slight abuse of notation by writing $V(x-y)\gamma(x,y)$ for the operator whose integral kernel is the function $(x,y)\mapsto V(x-y)\gamma(x,y)$. We remark that the number of particles $\tr(\gamma_t)$ and the energy $\cE_{\rm HF}(\gamma_t)$ are both conserved along the flow given by \eqref{eq:HF1}, as we will detail later on. Also, the appropriate set $\cK_{\rm HF}$ of initial data for the evolution equation \eqref{eq:HF1} will be defined below.

More specifically, we consider radial interaction potentials $V=V(|x|)$ that are essentially Newtonian at small and large distances. By this, we mean that $V$ is of the form
\begin{equation} 
\boxed{V(|x|)=-\frac{1}{|x|}+w(|x|) .} 
\label{form_V} 
\end{equation}
Here $w$ is a smooth and decaying function, satisfying the following conditions that we impose throughout this paper:
\begin{equation}
w\in C^1[0,\ii),\qquad r\,|w(r)|\leq C\quad\text{and}\quad (1+r^2)^{1+\epsilon}\,|w'(r)|\leq C
\label{assumption_w}
\end{equation}
for all $r\geq0$ and some $\epsilon>0$. Physically, one may think of $w$ as describing screening effects due to internal degrees of freedom of the gravitating fermions. Indeed, as remarked above, our arguments carry over to the case where spin is taken into account, and $w=w(|x|)$ would be a $q\times q$ matrix instead of a multiple of the identity (with an adequate redefinition of the HF-energy function $\cE_{\rm HF}(\gamma)$ above). 

\medskip
In the special case of $\gamma$ being an orthogonal projection of finite rank (i.\,e, we have $\gamma = \sum_{i=1}^N |\psi_i \rangle \langle \psi_i|$) and vanishing $w \equiv 0$, the existence of finite-time blowup for the time-dependent HF equation (\ref{eq:HF1}) was recently studied by Hainzl and Schlein in \cite{HaiSch-09}, by extending the arguments developed by Fr\"ohlich and Lenzmann in \cite{FroLen-07}. The main result of \cite{HaiSch-09} established that any (sufficiently regular and rapidly decaying) initial datum $\gamma$ with spherical symmetry (see Definition \ref{def:sph_sym_HF}) and negative energy $\cE_{\rm HF} (\gamma) < 0$ leads to finite-time blowup of the solution $\gamma_t$. Note that the assumption of finite rank was important for the method used in \cite{HaiSch-09} to work. In this paper, we present a new technical approach which allows us to prove blowup for $\gamma$ having infinite rank. Moreover, we show here that the existence of blowup solution is stable with respect to a certain class of perturbations $w$ of the Newton interaction, satisfying \eqref{assumption_w}. 

\medskip

After this preliminary remarks, we now state our results on the HF equation \eqref{eq:HF1}. To this end, we denote by $\gS_p$ the Schatten space (see \cite{Simon-79,ReeSim4}) of operators $A$ acting on $L^2 (\R^3,\C)$ such that $\norm{A}_{\gS_p}^p=\tr|A|^p<\ii$.  Further we introduce a set of fermionic density matrices
\begin{equation}
\cK_{\rm HF}:=\left  \{\gamma=\gamma^*\in\cX_{\rm HF} : 0\leq \gamma\leq 1 \right \},
\end{equation}
where the Sobolev-type space $\cX_{\rm HF}$ is defined by
\begin{equation}\label{eq:Hs} \cX_{\rm HF} := \left \{ \gamma \in \gS_1 : \| \gamma \|_{\cX_{\rm HF}} < \infty \right \} \end{equation} with the norm 
\begin{equation}
 \| \gamma \|_{\cX_{\rm HF}} := \norm{ (1-\Delta)^{1/4} \gamma (1-\Delta)^{1/4}}_{\gS_1} \, .
\end{equation}
In fact, it turns out that the initial-value problem (\ref{eq:HF1}) is locally well-posed in $\cK_{\rm HF}$.

\begin{theorem}[Well-posedness in Hartree-Fock theory]\label{thm:LWP} Suppose that $w$ satisfies \eqref{assumption_w}. For each initial datum $\gamma_0\in\cK_{\rm HF}$, there exists a unique solution $\gamma_t \in C^0 ([0,T), \cK_{\rm HF}) \cap C^1([0,T); \cX_{\rm HF}')$ with maximal time of existence $0 < T \leq \infty$. Moreover, we have conservation of energy and expected particle number:
\begin{equation*}
\cE_{\rm HF} (\gamma_t) = \cE_{\rm HF} (\gamma_0)\quad \text{ and }\quad \tr (\gamma_t) = \tr (\gamma_0) \qquad \mbox{for} \quad 0 \leq t < T.
\end{equation*}
Finally, if the maximal time satisfies $T < \infty$, then $\tr (-\Delta)^{1/2} \gamma_t \to \infty$ as $t \to T^-$.
\end{theorem} 
This theorem was proved in \cite{FroLen-07} in the special case $w=0$ and for $\gamma$ with finite rank. However, it is straightforward to extend the arguments of \cite{FroLen-07} to the case $w \neq 0$ and $\gamma$ with infinite rank; see also \cite{ChaGla-75, Chadam-76}. 

\begin{remark} \label{rem:lwp} {\em
i) By using the conservation laws, we can deduce that sufficiently small solutions $\gamma_t$ extend to all time. More precisely, there exists a universal constant $C_* > 0$ such that
$$
 \tr(\gamma_0) \kappa^{3/2} < C_* 
$$
implies global-in-time existence and a priori bounds in energy space; i.\,e., we have $T=+\ii$ and that $\| \gamma_t \|_{\cX_{\rm HF}} \leq C$ for all times $t \geq0$; see \cite{LenLew-08}. Physically, this global wellposedness result  states that white dwarfs of sufficiently small mass have a well-defined global dynamics and, in particular, no gravitational collapse can occur.

ii) Particular global-in-time solutions of \eqref{eq:HF1} are \emph{stationary states} satisfying $[H_\gamma,\gamma]=0$. Important examples for stationary states are given by the minimizers of $\cE_{\rm HF}(\gamma)$, subject to the constraint $\tr(\gamma)=\lambda$ with $\lambda\kappa^{3/2}$ not too large, which were proven to exist in \cite{LenLew-08}.

iii) Well-posedness can also be shown to hold true in Sobolev-type spaces $\cH^s$ with the norm $\| \gamma \|_{\cH^s} = \norm{ (1-\Delta)^{s/2} \gamma (1-\Delta)^{s/2}}_{\gS_1}$ for $s\geq1/2$, provided the initial datum $\gamma_0$ belongs to $\cH^s$. Indeed, we will make tacitly use of this fact about this persistence of higher regularity, guaranteeing the well-definedness of our calculations below.}
\end{remark}

We now define spherically symmetric states for which we will be able to prove blowup under some assumptions:
\begin{definition}[Spherically symmetric HF states]\label{def:sph_sym_HF} We say that $\gamma\in\cK_{\rm HF}$ is \emph{spherically symmetric} when $\gamma(Rx,Ry)=\gamma(x,y)$ for all $x,y\in\R^3$ and all $R\in SO(3)$. 
\end{definition}

It is easy to verify (see Lemma \ref{lem:conservation_spherical_sym} below) that spherically symmetry of $\gamma_t$ is preserved under the flow \eqref{eq:HF1}. We also note that, if $\gamma_0$ is sufficiently regular, then the condition of spherical symmetry can also be written as the commutator condition
\begin{equation}
[\gamma_t,L]=0 .
\end{equation} 
Here $L=-ix\wedge \nabla_x$ is the angular momentum operator, and $\wedge$ denotes the cross product on $\R^3$.

\medskip
Our main result of this section is the following blowup criterion. The proof will be given in Section \ref{sec:proof_HF} below.

\begin{theorem}[Blowup in Hartree-Fock theory]\label{thm:HF} Suppose that $w$ is a radial function satisfying \eqref{assumption_w} and that $\kappa>0$. Let $\gamma_0 \in \cK_{\rm HF}$ be spherically symmetric and suppose that
\[ \tr \, |x|^4 \gamma_0 + \tr \, (-\Delta) \gamma_0 + \tr \, |L|^2 \gamma_0 < \infty \]
where $L = -i x \wedge \nabla_x$ denotes the angular momentum operator.

Our conclusion is the following: If $\gamma_0$ has sufficiently negative energy, that is 
\begin{equation}\label{eq:en-neg} \cE_{\rm HF} (\gamma_0) < -\frac{\kappa}{2} \, \big(\tr (\gamma_0)\big)^2 \, \sup_{r\geq0} |w(r)+r\,w'(r)|_- \end{equation}
(where $|f|_-\geq0$ denotes the negative part of $f$), then the corresponding solution $\gamma_t$  to (\ref{eq:HF1}) blows up in finite time; i.\,e., we have $\tr \, (-\Delta)^{1/2} \, \gamma_t \to \infty \quad \mbox{as} \quad t \to T^{-}$ for some $T < \infty$.
\end{theorem}

\begin{remark} {\em By scaling, it is simple to show the existence of $\gamma_0 \in \cK_{\rm HF}$ satisfying the condition (\ref{eq:en-neg}) for sufficiently large coupling constants $\kappa$, or equivalently for a sufficiently high particle number $\tr(\gamma_0)$. Note also that $\tr(\gamma_0)\kappa^{3/2}$ must be sufficiently large, since otherwise $T=+\ii$ holds, as mentioned in Remark \ref{rem:lwp} above.}
\end{remark}

Regarding the proof of Theorem \ref{thm:HF}, we remark that we follow a virial-type argument developed in \cite{FroLen-07b,FroLen-07} for the Hartree equation, which was recently extended in \cite{HaiSch-09} to Hartree-Fock equations, by using an additional conversation law for the square of the angular momentum $\tr(|L|^2 \gamma_t)$ for radial solutions. More specifically, it was shown in \cite{HaiSch-09} that spherically symmetric solutions $\gamma_t$ of the Hartree-Fock equation \eqref{eq:HF1} with $V=-1/|x|$ satisfy the inequality
\begin{equation}
\tr(M \gamma_t) \leq 2 \mathcal{E}_{\rm HF}(\gamma_0) + C \tr(\gamma_0) \tr ( (1+|L|^2 \gamma_0) t + C,
\end{equation}
provided that $\gamma_0$ is finite-rank. Since the relativistic ``virial'' operator 
$$M  = \sum_{j=1}^3 x_i \sqrt{-\Delta + m^2} x_i$$ 
is nonnegative, this yields a finite maximal time of existence $T < \infty$ and hence the desired blowup result. In the arguments developed below, we remove the finite-rank condition, by expanding the solution $\gamma_t$ in terms of sectors of angular momentum. This approach is also essential when addressing the blowup for HFB theory, as done in the next section. However, the conservation law for $\tr(|L|^2 \gamma_t)$, which plays an essential role in the argument sketched above, is no longer at our disposal in HFB theory, as will be detailed below.

\section{Blowup in Hartree-Fock-Bogoliubov theory} \label{sec:HFB}
In this section, we now turn to the Hartree-Fock-Bogoliubov (HFB) model, which generalizes the well-known Hartree-Fock model studied in the previous section. It is fair to say that HFB theory is a widely used tool for (partially) attractive quantum systems, which are found -- for example -- in nuclear and stellar physics; see \cite{RinSch-80,BenHeeRei-03,DeaHjo-03} for a detailed exposition. As a notable fact, we mention that HFB yields (in some limiting approximation) the well-known Bardeen-Cooper-Schrieffer (BCS) theory, which plays a fundamental role for the understanding of superconductivity and superfluidity. 

The outline of this section is briefly as follows. First, we introduce the HFB energy and its associated time-dependent equations. Second, we point out the key difficulties in the blowup analysis encountered within HFB theory. Finally, we state our main result on finite-time blowup for the HFB model with an angular momentum cutoff; see Theorem \ref{thm:HFB_cut_off} below.

\subsection{HFB Theory: Preliminaries and Difficulties}

\subsubsection{HFB Energy}
The HFB model involves two variables: the same one-body density matrix $\gamma$ as in HF theory, and an additional two-body fermionic wavefunction $\alpha(x,y)\in L^2(\R^3;\C)\wedge L^2(\R^3;\C)$\footnote{Here $\wedge$ denotes the antisymmetric tensor product.} describing the so-called \emph{Cooper pairs}. The amount of Cooper pairing is given by  $\iint|\alpha(x,y)|^2dx\,dy$; and it must vanish when $\gamma^2=\gamma$ is a pure Hartree-Fock state. For simplicity, we will always use the same notation for the two-body wavefunction $\alpha$ and the corresponding (non self-adjoint but compact) operator acting on the one-body space $L^2(\R^3;\C)$, whose integral kernel is the function $(x,y)\mapsto\alpha(x,y)$. Note that the assumption that $\alpha$ is fermionic can then be expressed as $\alpha^T=-\alpha$. We remind the reader that again we have neglected any spin degrees of freedom to simplify our presentation. 

Within the HFB framework, the operators $\gamma$ and $\alpha$ are, by assumption, linked through the following operator constraint:
\begin{equation}
\left(\begin{matrix}
0 & 0\\
0 &  0
\end{matrix}\right)\leq \left(\begin{matrix}
\gamma & \alpha\\
\alpha^* &  1-\overline{\gamma}
\end{matrix}\right)\leq \left(\begin{matrix}
1 & 0\\
0 &  1
\end{matrix}\right)\quad\text{on } L^2(\R^3;\C)\oplus L^2(\R^3;\C).
\label{constraint_2}
\end{equation}
In fact, this inequality guarantees that the pair $(\gamma,\alpha)$ is associated with a unique quasi-free state in Fock space; see \cite{BacLieSol-94}. The average number of particles is given (as in HF theory) by the trace $\tr(\gamma)$.

The HFB energy functional, for a system interacting through a radial two-body potential $V$, is given by
\begin{equation}
\cE_{\rm HFB}(\gamma,\alpha) :=\cE_{\rm HF}(\gamma)+\frac{\kappa}{2} \iint_{\R^3\times\R^3}{|\alpha(x,y)|^2}V(|x-y|) \, dx\,dy.
\label{def:energy_HFB}
\end{equation}
It is then obvious that a ground state (if it exists) always has $\alpha\equiv0$ in the case of a purely repulsive potential $V\geq0$. Therefore, the use of HFB theory is only meaningful when $V$ is not everywhere positive. In fact, in this case, it is often observed that there is a phase transition from an HF ground state to an HFB ground state when the coupling constant $\kappa$ is increased.

As in the previous section, we consider radial potentials $V=V(|x|)$ of the form \eqref{form_V} above. Furthermore, we remark that the HFB energy $\cE_{\rm HFB}(\gamma, \alpha)$ can be appropriately defined (see \cite{LenLew-08}) on the set of generalized density matrices:
\begin{equation}
\cK_{\rm HFB} =  \left\{(\gamma,\alpha) \in \cX_{\rm HFB}\ :\ \left(\begin{matrix}
0&0\\ 0&0\end{matrix}\right) \leq \left(\begin{matrix}
\gamma&\alpha\\ \alpha^*&1-\overline{\gamma}\end{matrix}\right)\leq\left(\begin{matrix}
1&0\\ 0&1\end{matrix}\right) \right\}.
\end{equation}
The space $\cX_{\rm HFB}$ is the correct energy space for generalized density matrices, which is defined as 
\begin{equation}
\cX_{\rm HFB} = \big  \{(\gamma,\alpha) \in \gS_1\times\gS_2 \ : \ \gamma^* = \gamma,\ \alpha^T=-\alpha,\ \|(\gamma,\alpha) \|_{\cX_{\rm HFB}}<\ii\big\},
\label{def_set_X}
\end{equation}
and is equipped with the norm
\begin{equation}
\|(\gamma,\alpha) \|_{\cX_{\rm HFB}} = \big \| (1-\Delta)^{1/4} \gamma (1-\Delta)^{1/4} \big\|_{\gS_1} + \big \| (1-\Delta)^{1/4} \alpha \big \|_{\gS_2} .
\end{equation}
We remind the reader that $\gS_1$ and $\gS_2$ denote the space of trace-class and Hilbert-Schmidt operators on $L^2(\R^3; \C)$, respectively. Also, note that $\| (1-\Delta)^{1/4} \alpha \|_{\gS_2}$ is equivalent to the Sobolev norm $\| \alpha(\cdot, \cdot) \|_{H^{1/2}(\R^3 \times \R^3)}$ where $\alpha(x,y)$ is regarded as a two-body wavefunction defined on $\R^3 \times \R^3$. 

The existence of minimizers for the HFB energy $\cE_{\rm HFB}(\gamma, \alpha)$, under the constraint that $\tr\gamma=\lambda$ for $\lambda \kappa^{3/2}$ not too large, was given in \cite{LenLew-08}. Moreover, some important properties were derived in \cite{BacFroJon-08,LenLew-08}. We note it is not known whether a ground state always has $\alpha\neq0$ for $V$ taking the form \eqref{form_V}, but this is expected to hold true (at least for sufficiently large local potentials $w$).

In view of \eqref{constraint_2} and following \cite{BacLieSol-94,LenLew-08}, we introduce a generalized one-body density matrix
\begin{equation}
\Gamma=\left(\begin{matrix}
\gamma&\alpha\\ \alpha^*&1-\overline{\gamma}\end{matrix}\right), 
\label{def_GAMMA}
\end{equation}
as well as the corresponding \emph{HFB mean-field operator} 
\begin{equation}
F_{\Gamma}:= \left(\begin{matrix}
H_\gamma & \Pi_\alpha\\
\Pi_\alpha^* & -\overline{H_\gamma}\\
\end{matrix}\right)
\label{def_H_SCF2}
\end{equation}
which acts on $L^2(\R^3; \C)\oplus L^2(\R^3; \C)$, with $\overline{H_\gamma}$ denoting the conjugate of $H_\gamma$. Here 
\begin{equation}
H_\gamma:= K+\kappa(V\ast\rho_\gamma)(x)-\kappa R_\gamma(x,y)
\label{def_H_SCF}
\end{equation}
denotes the Hartree-Fock mean-field operator introduced before in \eqref{def:mean-field}; and $\Pi_\alpha$ is given by its kernel
$$\Pi_\alpha(x,y)=\kappa V(x-y)\alpha(x,y).$$
Moreover, it turns out to be convenient to define the number operator as
\begin{equation}
\cN=\left(\begin{matrix}
1&0\\ 0&-1\end{matrix}\right).
\label{def_N}
\end{equation}
Note that $\Gamma$ commutes with $\cN$ if and only if $\alpha=0$, i.\,e., if and only if the corresponding quasi-free state in Fock space also commutes with the number operator $\cN$; see, e.\,g., \cite{BacLieSol-94}. 

\begin{remark} {\em
It was shown in \cite{BacLieSol-94,LenLew-08} that any minimizer of the HFB energy \eqref{def:energy_HFB} solves the following self-consistent equation:
 \begin{equation}
\Gamma=\chi_{(-\ii,0)}\left(F_\Gamma-\mu \cN\right)+D ,
\label{SCF_GAMMA} 
\end{equation}
for some Lagrange multiplier $\mu<0$ chosen to ensure the constraint $\tr(\gamma)=\lambda$, and where $D$ is a finite rank operator of the same matrix form as $\Gamma$, satisfying ${\rm ran} (D)\subset \ker (F_\Gamma-\mu \cN)$. It was also shown in \cite{LenLew-08,BacFroJon-08} that if $\alpha\neq0$ and $w \equiv0$, then $\gamma$ must be infinite rank.
}
\end{remark}

\subsubsection{The HFB evolution equation}
Let us now turn to the time-dependent HFB theory. This evolution equation can be concisely written in terms of the generalized density matrix $\Gamma=\Gamma(\gamma,\alpha)$ introduced in \eqref{def_GAMMA}. The corresponding initial-value problem reads as follows:
\begin{equation}
\boxed{\left\{\begin{array}{l}
\displaystyle i\frac{d}{dt}\Gamma_t=\left[F_{\Gamma_t}\,,\, \Gamma_t\right] \\[1ex]
(\gamma_{|t=0},\alpha_{|t=0})=(\gamma_0,\alpha_0)\in\cK_{\rm HFB}.
\end{array}\right.}
\label{eq:time-dep_HFB}
\end{equation}
This evolution equation can also be regarded as obtained from ``forcing'' the many-body Schr\"odinger evolution in Fock space to stay on the manifold of HFB states, by means of the Dirac-Frenkel principle in quantum mechanics. With regard to the wellposedness of its initial-value problem, we note that a straightforward adaptation of \cite{ChaGla-75,Chadam-76,BovPraFan-76,FroLen-07,HaiLewSpa-05,Lenzmann-07} yields the existence and uniqueness of a maximal solution. We record the wellposedness fact as follows.
\begin{theorem}[Well-posedness in Hartree-Fock-Bogoliubov Theory]
Assume that $w$ is a radial function satisfying \eqref{assumption_w} and that $0\leq\kappa<4/\pi$ holds. For each initial datum $(\gamma_0, \alpha_0) \in \cK_{\rm HFB}$, there exists a unique solution $(\gamma_t,\alpha_t) \in C^0([0,T); \cK_{\rm HFB}) \cap C^1([0,T); \cX_{\rm HFB}')$ solving \eqref{eq:time-dep_HFB} with maximal time of existence $0 < T \leq \infty$. Moreover, we have conservation of  energy and expected number of particles, i.\,e.,
\[
 \cE(\gamma_t,\alpha_t) = \cE(\gamma_0, \alpha_0) \quad \mbox{and} \quad \tr (\gamma_t) = \tr(\gamma_0) \qquad \mbox{for} \quad 0 \leq t < T,
\]
as well as the a priori bound for the pairing wave function:
\[
\tr(\alpha_t^* \alpha_t) \leq C \qquad \mbox{for} \quad 0 \leq t < T.
\]
Finally, if $T<\ii$ holds, then we have $\tr(-\Delta)^{1/2}\gamma_t\to\ii$ as $t\to T^-$.
\end{theorem}

\begin{remarks} {\em
i) It was shown in \cite{LenLew-08} that when $\tr(\gamma_0)\kappa^{3/2}$ is sufficiently small, then one has $T=+\ii$ and $(\gamma_t,\alpha_t)$ is bounded in $\cX_{\rm HFB}$.

ii) Note the condition $\kappa < 4/\pi$ as opposed to HF theory; see Section \ref{sec:ang:HFB} for more details on this.

iii) It is straightforward to verify that if the generalized matrix $\Gamma_0$ associated with the initial condition $(\gamma_0,\alpha_0)$ is an orthogonal projector on $L^2(\R^3;\C)\oplus L^2(\R^3;\C)$, then so is $\Gamma_t$ for all times $t\in[0,T)$. Similarly, if $\alpha_0=0$ (i.\,e., the initial datum is a Hartree-Fock state) then we have $\alpha_t\equiv0$ for all time. Therefore the equation \eqref{eq:time-dep_HFB} is a generalization of the HF evolution equation \eqref{eq:HF1}.

iv) The {\em a priori} control on $\tr(\alpha^*_t \alpha_t)$ simply follows from the operator inequality $\gamma^2 + \alpha^* \alpha \leq \gamma$ (since $\Gamma^2 \leq \Gamma$) and the conservation of $\tr(\gamma)$. See Section \ref{sec:ang:HFB} for more details on the time evolution of $\alpha$.

v) As for stationary solutions of \eqref{eq:HF2}, we remark that it can be verified that if $\Gamma_0$ satisfies $[F_{\Gamma_0}-\mu\cN,\Gamma_0]=0$ for some $\mu$, then $\Gamma(t)=e^{-i\mu \cN t}\,\Gamma_0\,e^{i\mu \cN t}$ is solution of equation \eqref{eq:time-dep_HFB}. In particular, we see that HFB minimizers give rise to stationary solutions of \eqref{eq:time-dep_HFB}, as one naturally expects. 
}
\end{remarks}

\begin{remark}[Spin symmetry]\label{rmk:spin_symmetry}{\em
There is a very natural spin symmetry in HFB theory, which was already used before in minimization problems for the Hubbard model in \cite{BacLieSol-94} and for general systems with an interaction potential of negative type in \cite{BacFroJon-08}. In the case of the Newtonian interaction, this was used in \cite{LenLew-08}.

Let us assume only in this remark that the problem is now posed in $L^2(\R^3;\C^2)$ and let us introduce the unitary operators for $k=1,2,3$
$$\Sigma_k=\left(\begin{matrix}
i\sigma_k & 0\\ 0 & \overline{i\sigma_k }
\end{matrix}\right)=i\left(\begin{matrix}
\sigma_k & 0\\ 0 & -\overline{\sigma_k }
\end{matrix}\right)$$
where $\sigma_k$ are the Pauli matrices.
A simple calculation shows that $[\Sigma_k,\Gamma]=0$ is equivalent to $[\sigma_k,\gamma]=0$ and $\sigma_k \alpha+\alpha\overline{\sigma_k}=0$. It can then be checked\footnote{The condition $\Sigma_k\Gamma\Sigma_k^*=\Gamma$ for $k=1,2,3$ is indeed equivalent to $\gamma_{\uparrow\downarrow}=\gamma_{\downarrow\uparrow}=0$, $\gamma_{\uparrow\uparrow}=\gamma_{\downarrow\downarrow}$ and $\alpha\equiv0$, as may be seen by a simple computation. We do not want to impose such a condition.} that $[\Sigma_k,\Gamma]=0$ for $k=1,2$ if and only if 
$\gamma_{\uparrow\downarrow}=\gamma_{\downarrow\uparrow}=0$, $\gamma_{\uparrow\uparrow}=\gamma_{\downarrow\downarrow}$, $\alpha_{\uparrow\uparrow}=\alpha_{\downarrow\downarrow}=0$ and $\alpha_{\downarrow\uparrow}=-\alpha_{\uparrow\downarrow}$.

Let us now assume that our initial state $\Gamma_0$ takes the special form
\begin{equation}
\Gamma_0=\left(\begin{matrix}
\gamma_0 & \alpha_0\\
\alpha_0^* & 1-\overline{\gamma}_0
\end{matrix}\right)\quad\text{with}\quad \gamma_0=\tau_0\otimes\left(\begin{matrix}
1 & 0\\ 0 & 1\end{matrix}\right)\quad\text{and}\quad \alpha_0=a_0\otimes\left(\begin{matrix}
0 & 1\\ -1 & 0\end{matrix}\right)
\label{general_form_spin_Gamma} 
\end{equation}
where $a_0$ and $\tau_0$ act on $L^2(\R^3;\C)$ with $\tau_0
^*=\tau_0$ and $a_0^T=a_0$. If $w$ is a multiple of the $2\times2$ identity matrix,  then both the kinetic energy $K$ and our interaction potential do not depend on the spin. It is then clear that we have $\Sigma_kF_\Gamma\Sigma_k=F_{\Sigma_k\Gamma\Sigma_k^*}$ for $k=1,2,3$. Hence we see that $\Sigma_k\Gamma_t\Sigma_k^*\in\cK_{\rm HFB}$ solves the same time-dependent equation as $\Gamma_t$. By uniqueness we deduce that $\Gamma_t$ has the same form as $\Gamma_0$ for all times. Therefore if one seeks for solutions of the above special form, one can reduce the spin-$1/2$ case to a no-spin model with the condition $\alpha^T=-\alpha$ replaced by a singlet-type condition $a^T=a$. All our results also hold true under this constraint. A similar reduction can be done for a system with $q \geq 1$ degrees of freedom.

The special form \eqref{general_form_spin_Gamma} is sometimes called \emph{time-reversal symmetry} as the time reversal symmetry operators are usually defined through $T_k=e^{i\sigma_k\pi/2}K=\Sigma_kK$, where $K$ is the conjugation (anti-linear) operator. Indeed, one verifies that $T_k\Gamma_tT_k^{-1}$ solves the backward time-dependent equation. }
\end{remark}

\subsubsection{HFB: The Difficulties} \label{sec:ang:HFB}
In the previous section we have studied the blowup of solutions in the Hartree-Fock case. As we will explain now, the HFB generalization brings in new substantial difficulties as follows.
\begin{itemize}
\item The pairing term coming from $\alpha_t$ is difficult to control.
\item Powers of the angular momentum $\tr(|L|^s \gamma_t)$ are generally {\em not conserved} along the HFB flow, even for spherical symmetric solutions $(\gamma_t, \alpha_t)$ defined below. This is a striking contrast to HF theory. 
\end{itemize}
It is enlighting to rewrite the time-dependent equation \eqref{eq:time-dep_HFB} as a coupled system with a von Neumann-type equation for the one-body density matrix $\gamma_t$ and a two-body Schr\"odinger equation for the pairing wavefunction $\alpha_t$. In fact, a simple calculation yields the following system:
\begin{equation}
 \boxed{\left\{\begin{array}{rl}
\displaystyle i\,\frac{d}{dt}\gamma_t &= \left[H_{\gamma_t}\, ,\, \gamma_t\right] + iG_{\alpha_t},\\
\displaystyle i\, \frac{d}{dt}\alpha_t(x,y) & = \displaystyle  \bigg(\left(H_{\gamma_t}\right)_x+\left(H_{\gamma_t}\right)_y+\kappa V(|x-y|)\\
& \displaystyle \qquad\qquad\qquad-\kappa\big((\gamma_t)_x+(\gamma_t)_y\big)V(|x-y|)\bigg)\,\alpha_t(x,y).
\end{array}\right.}
\label{eq:system_TD_HFB}
\end{equation}
The time-dependent equation for $\gamma_t$ contains the trace-class coupling term $G_\alpha$ arising from the pairing wavefunction $\alpha$. By an elementary calculation, this term is found to be
$G_\alpha =i(\Pi_\alpha\alpha^*-\alpha\Pi_\alpha^*)$ or, in terms of its integral kernel,
$$G_\alpha(x,y)=i\,\kappa\int_{\R^3}\alpha(x,z)\overline{\alpha(y,z)}\left(V(y-z)-V(x-z)\right)\,dz.$$
Note that $ \tr(\gamma_t)$ is constant along the trajectories, because of $\tr(G_\alpha)=0$.

It is fair to say that the two-body wavefunction $\alpha_t$ evolves through a time-dependent two-body Schr\"odinger equation with Hamiltonian
\begin{equation}
\left(H_{\gamma_t}\right)_x+\left(H_{\gamma_t}\right)_y+\kappa V(|x-y|)-\kappa\big((\gamma_t)_x+(\gamma_t)_y\big)V(|x-y|)
\label{def:Hamil_alpha}
\end{equation}
Here, the first two terms in this Hamiltonian form a one-body operator arising from the mean-field operator $H_\gamma$, which itself contains the potentials induced by $\gamma$ as external sources. The third term $\kappa V(|x-y|)$ is a usual two-body Newtonian interaction term (perturbed by our local interaction $w$). The last term 
$\kappa\big((\gamma_t)_x+(\gamma_t)_y\big)V(|x-y|)$ in \eqref{def:Hamil_alpha} is a non-self-adjoint (but compact) operator 
which is responsible for the fact that $\alpha_t$ does not have a constant $L^2$ norm along a given trajectory (except of course if $\alpha_t\equiv0$). We note that since $0\leq \kappa<4/\pi$, the two-body operator \eqref{def:Hamil_alpha} is actually equal to a well-defined self-adjoint operator, plus a compact non-self-adjoint operator. Therefore, the two-body equation for $\alpha_t(x,y)$ is well-defined.

\medskip

Let us now detail the behavior of moments of angular momentum along the HFB flow. We recall for the reader's orientation that, in HF theory, higher moments of angular momentum are conserved, provided the initial state is spherically symmetric. This fact is one of the main tools used in blowup proof for HF, as we have already seen. First, we define spherically symmetric HFB states in a similar fashion as in the HF case.

\begin{definition}[Spherically symmetric HFB states]\label{def:sph_sym_HFB} We say that $(\gamma,\alpha)\in\cK_{\rm HFB}$ is \emph{spherically symmetric} when $\gamma(Rx,Ry)=\gamma(x,y)$ and $\alpha(Rx,Ry)=\alpha(x,y)$ for all $x,y\in\R^3$ and all $R\in SO(3)$. 
\end{definition}

When the HFB initial datum $(\gamma_0,\alpha_0)$ is spherically symmetric, it can easily be seen that $(\gamma_t,\alpha_t)$ stays spherically symmetric on its time interval of existence. In terms of the angular momentum operator (provided that $\gamma$ and $\alpha$ are smooth enough) the condition of spherical symmetry may also be written as 
\begin{equation}
[L,\gamma_t]=0 \quad \mbox{and}  \quad (L_x+L_y)\alpha_t(x,y)=0.
\end{equation}
Moreover, note also that we find that $[|L|^2,\gamma_t]=0$ holds. However, we do {\em not} obtain any other information about the two-body wavefunction $\alpha_t$. In particular, we do \emph{not} have that $(|L_x|^2+|L_y|^2)\alpha_t(x,y)=0$, which would mean that $\alpha_t(Rx,y)=\alpha_t(x,Ry)=\alpha_t(x,y)$ for all rotations $R\in SO(3)$. 

Let us elaborate more explicitly the fact that $\tr (|L|^2 \gamma_t)$ is not conserved for spherically symmetric solutions in HFB theory. To this end, we calculate the variation of the expectation value of $|L|^2$, assuming that $[L,\gamma_t]=0$ and $(L_x+L_y)\alpha_t=0$ for all times $t\in[0,T)$. We obtain
\begin{multline}
\frac{d}{dt}\tr(|L|^2\gamma_t)\\=\frac{\kappa\, i}{2}\pscal{\alpha_t,\left[|L_x|^2+|L_y|^2\,,\, \frac{1}{|x-y|}-w(|x-y|)\right]\alpha_t}_{L^2(\R^3;\C)\otimes L^2(\R^3;\C)}. 
\end{multline}
As $|x-y|^{-1}$ and $w(|x-y|)$ do not commute with $|L_x|^2+|L_y|^2$, the above term is usually not zero, and hence the square of the angular momentum is not conserved. This is a major difficulty that prevents us to use the same proof as in the HF case!

To see that the square of the angular momentum does actually vary in time, let us consider a sufficiently smooth initial datum $(\gamma_0,\alpha_0)$ with $\alpha_0\neq0$ such that $\gamma_0(x,y)=\gamma_0(|x|,|y|)$ and $\alpha_0(x,y)=\alpha_0(|x|,|y|)$. The latter property means exactly that $(|L_x|^2+|L_y|^2)\alpha_0(x,y)=0$. It is clear from the above formula that the derivative of $\tr(|L|^2\gamma_t)$ vanishes at time $t=0$. We have the following series of equality:
\begin{align*}
&\left(|L_x|^2+|L_y|^2\right)\left((\gamma_0)_x+(\gamma_0)_y\right)\alpha_0(|x|,|y|)V(x-y)\\
&\qquad\qquad\qquad\qquad=\left((\gamma_0)_x+(\gamma_0)_y\right)\left(|L_x|^2+|L_y|^2\right)\alpha_0(|x|,|y|)V(x-y)\\
&\qquad\qquad\qquad\qquad=2\left((\gamma_0)_x+(\gamma_0)_y\right)L_x\cdot L_y\alpha_0(|x|,|y|)V(x-y)\\
&\qquad\qquad\qquad\qquad=2\left((\gamma_0L)_x\cdot L_y+L_x\cdot (\gamma_0L)_y\right)\alpha_0(|x|,|y|)V(x-y)\\
&\qquad\qquad\qquad\qquad=0
\end{align*}
since $\gamma_0 L=L\gamma_0=0$ and $(L_x+L_y)V(x-y)=0$. Using this, we can compute the second derivative of the expectation value of $|L|^2$. We obtain
\begin{multline}
\frac{d^2}{dt^2}\tr(|L|^2\gamma_t)_{|t=0}=\kappa^2\pscal{\frac{\alpha_0(|x|,|y|)}{|x-y|},\left(|L_x|^2+|L_y|^2\right)\frac{\alpha_0(|x|,|y|)}{|x-y|}}_{L^2\otimes L^2}\\
+\kappa^2\pscal{\beta_0(x,y),\left(|L_x|^2+|L_y|^2\right)\beta_0(x,y)}_{L^2\otimes L^2}
\label{2nd_derivative_L2}
\end{multline}
where $\beta_0(x,y)=\alpha_0(|x|,|y|)w(|x-y|)$. It is easy to see that \eqref{2nd_derivative_L2} is always a positive quantity, except when $\alpha_0=0$. 

To sum up: Even when we start at time $t=0$ with a state in the zero angular momentum sector, it immediately acquires a change of the square of angular momentum, in stark contrast to HF theory \cite[p. 743]{FroLen-07}. 


\subsection{Blowup with Angular Momentum Cutoff}
As we have explained, the second moment of the angular momentum, $\tr(|L|^2\gamma_t)$, is not conserved for spherical states in general. In fact, it seems quite difficult to get any useful bound on this quantity. There is yet another difficulty that we have not mentioned so far: the pairing term is difficult to control and our proof reveals that we actually need a bound on $\tr(|L|^{6+\epsilon}\gamma_t)$ for some $\epsilon>0$; see Lemma \ref{lem:estim_alpha} below. It seems a very interesting question to investigate whether the angular momentum can indeed substantially alter the collapsing behavior of a gravitational system in HFB theory, or actually prevent collapse at all (which, however, is very unlikely  in our opinion).

For all these reasons, we will now introduce a simpler model for which we are able to prove blowup of spherical solutions with sufficiently negative energy. The idea is to enforce an {\em angular momentum cutoff} in the model as follows. Let $\Lambda \geq 0$ be a fixed integer, and define the (infinite-dimensional) orthogonal projector\footnote{As usual $\chi_{A}(x)$ is the characteristic function of the set $A$.}
$$P_\Lambda:=\chi_{[0,\Lambda(\Lambda+1)]}\left(|L|^2\right),$$
which projects $L^2(\R^3;\C)$ onto sectors with angular momentum $\ell \leq \Lambda(\Lambda+1)$. We will now \emph{impose} that our density matrices $\gamma$ and $\alpha$ satisfy for all times
\begin{equation}
P_\Lambda\gamma=\gamma P_\Lambda=\gamma,\qquad (P_\Lambda\otimes P_\Lambda)\,\alpha(x,y)=\alpha(x,y).
\label{HFB_mom_cut_off} 
\end{equation}
Note that the assumption on $\alpha$ is equivalent to $\big((P_\Lambda)_x+(P_\Lambda)_y\big)\alpha(x,y)=2\alpha(x,y)$, or in operator terms to $P_\Lambda\alpha=\alpha P_\Lambda=\alpha$.
Our assumption \eqref{HFB_mom_cut_off} is the same as saying that we restrict ourselves to HFB states in Fock space $\cF$ that actually belong to the restricted Fock space 
$$\cF_\Lambda=\oplus_{N\geq0}\wedge_1^N\gH_\Lambda\subset\cF,$$ 
where $\gH_\Lambda=P_\Lambda L^2(\R^3;\C)$. Hence, introducing the angular momentum cutoff is the same as replacing the one-body space $L^2(\R^3;\C)$ by $\gH_\Lambda$. This procedure does not change the HFB energy $\cE_{\rm HFB}$. But  it changes the HFB evolution equation, because the constraint \eqref{HFB_mom_cut_off} needs to be conserved for all times. The time-dependent equation in the cutoff model then becomes
\begin{equation}
\boxed{\left\{\begin{array}{l}
\displaystyle i\frac{d}{dt}\Gamma_t=\left[\cP_\Lambda F_{\Gamma_t}\cP_\Lambda \,,\, \Gamma_t\right] \\
(\gamma_{|t=0},\alpha_{|t=0})=(\gamma_0,\alpha_0)\in\cK_{\rm HFB} \text{ satisfying \eqref{HFB_mom_cut_off}},
\end{array}\right.}
\label{eq:time-dep_HFB_cut_off}
\end{equation}
where
$$\cP_\Lambda:=\left(\begin{matrix}
P_\Lambda & 0\\ 0 & 1-P_\Lambda\end{matrix}\right).$$
The coupled system \eqref{eq:system_TD_HFB} can be rewritten in a similar fashion.

Let us recall that the HFB dynamics introduced in the last section was obtained by restricting the (formal) unitary evolution of the Hamiltonian in Fock space to the submanifold of HFB states, by means of the Dirac-Frenkel principle. By introducing a cutoff in angular moment, we simply restrict ourselves to the time-evolution to the even smaller manifold of states belonging to $\gH_\Lambda$. We remark that, in numerical simulations, an angular momentum cutoff (as we introduced) will be in fact always imposed.

\medskip
The following main result of this paper now establishes blowup in HFB theory with an angular momentum cutoff. 

\begin{theorem}[Blowup for HFB with angular momentum cutoff]\label{thm:HFB_cut_off} Suppose $w$ satisfies assumption \eqref{assumption_w}.
Let $(\gamma_0,\alpha_0)\in\cK_{\rm HFB}$ be a radially symmetric initial datum satisfying \eqref{HFB_mom_cut_off} and such that
\[ \tr \, |x|^4 \gamma_0 + \tr \, (-\Delta) \gamma_0  < \infty. \]
Our conclusion is the following: If $(\gamma_0,\alpha_0)$ has sufficiently negative energy, that is
\begin{equation}
\cE_{\rm HFB} (\gamma_0,\alpha_0) < - \frac\kappa2 \sup_{r\geq0}|w(r)+rw'(r)|_-\left(\big(\tr\gamma_0\big)^2+\tr\gamma_0\right) ,
\label{condition_blow_up_HFB_cut_off}
\end{equation}
then the unique maximal local-in-time solution $(\gamma_t,\alpha_t)$ of \eqref{eq:time-dep_HFB_cut_off} blows up in finite time, i.\,e., we have $T<\ii$ and $\tr(-\Delta)^{1/2}\gamma_t\to\ii$ as $t\to T^-$.
\end{theorem}

The proof of Theorem \ref{thm:HFB_cut_off} is provided below in Section \ref{sec:proof_HFB}.

\begin{remark}[Removing the angular momentum cutoff]
{\em It is not obvious how to get any information on the initial model from the study of the model with an angular momentum cutoff. However, it is easy to grasp from our proof in Section \ref{sec:proof_HFB} that if $\tr(|L|^{6+\epsilon}\gamma_t)$ stays bounded (or does not grow too fast), then blowup must occur under the same condition \eqref{condition_blow_up_HFB_cut_off}.}
\end{remark}

\begin{remark}[Case of the Laplacian]
{\em
We mention, as a side remark, that in a nonrelativistic HFB model, where $K$ is replaced by the Laplacian $-\Delta$, and with the scaling critical two-body potential $V(x-y)=-1/|x-y|^2$, higher powers of the angular momentum are also \emph{a priori} not conserved. However, in the nonrelativistic setting, one can use in the proof the observable $M=|x|^2$, instead of the operator $M=\sum_{i=1}^3x_i\sqrt{m^2-\Delta}x_i$, as done below. Hence, for nonrelativistic systems, no angular momentum conservation is needed, and one easily sees that blowup always occurs under a negative energy condition, even for non-spherically symmetric initial data. 

Nevertheless, it should be mentioned that the Laplacian $-\Delta$ is qualitatively very different from the pseudo-relativistic operator $\sqrt{-\Delta+m^2}$, which models a relativistic system with a `finite speed of propagation'.}
\end{remark}


\section{Proofs} \label{sec:proofs}
\subsection{Preliminaries}
In the proofs of our main results, we shall need several auxiliary lemmas, which we gather in this section for convenience. 

\subsubsection{Spherical Harmonics Expansion}
An important tool will be the expansion of the time-dependent states in terms of spherical harmonics. As we restrict ourselves to spherically symmetric states, our states will be uniformly distributed in each angular momentum sector and we will only encounter Legendre polynomials.

We recall that the $\ell$-th Legendre polynomial $P_\ell(t)$ is a real polynomial of degree $\ell$ defined on $[-1;1]$ which satisfies $-1\leq P_\ell(t)\leq 1$ and $P_\ell(1)=1$, for all integers $\ell\geq0$. Its is linked to the usual spherical harmonics $Y^m_\ell$ by the formula
\begin{equation}
P_\ell(\omega\cdot\omega')=\frac{1}{2\ell+1}\sum_{m=-\ell}^\ell Y_\ell^{m}(\omega)\overline{Y_\ell^{m}(\omega')},
\label{def:Legendre} 
\end{equation}
where $\omega,\omega'\in \mathbb{S}^2 =\{ x \in \R^3 : |x| = 1\}$. 

We define for all $\ell,\ell'$ the following function:
\begin{align}
F_{\ell,\ell'}(r,r') &= 2\pi \int_{-1}^1 P_\ell(t)P_{\ell'}(t)\, V \left( \sqrt{r^2 + r'^2 - 2 rr' t } \right) \, dt\nonumber\\
&= 2\pi \int_{-1}^1 P_\ell(t)P_{\ell'}(t)\left(-\frac{1}{\sqrt{r^2 + r'^2 - 2 rr' t }}+w\left(\sqrt{r^2 + r'^2 - 2 rr' t }\right) \right) \, dt.\label{def:F_ell}
\end{align}
We recall that $V(x)=-1/{|x|}+w(|x|)$. With the usual abuse of notation, we will henceforth write $V(|x|)$ as well as $V(x)$, whichever is more convenient for the case at hand.

The function $F_{\ell,\ell'}$ will often appear when integrating out the angle variable in the interaction potential between the projections of our states in, respectively, the $\ell$-th and the $\ell'$-th angular momentum sector. We will need the following
\begin{lemma}[Estimates on $F_{\ell,\ell'}$]\label{lem:estim_F_ell}
We have the following bounds:
\begin{equation}
|F_{\ell,\ell'}(r,r')|\leq \frac{4\pi\big(1+\norm{rw(r)}_{L^\ii[0,\ii)}\big)}{\max(r,r')}
\label{estim_F_ell_sup}
\end{equation}
and
\begin{equation}
| r^2\, \partial_r\,F_{\ell,\ell'}(r,r')|\leq C\left(1+\ell^2+(\ell')^2\right)+C_\epsilon\norm{(1+r^2)^{1+\epsilon}w'(r)}_{L^\ii[0,\ii)}
\label{estim_F_ell_diff_sup}
\end{equation} 
for some universal constant $C$ and some constant $C_\epsilon$ depending only on $\epsilon$.
\end{lemma}

\begin{proof}
Clearly we have
$$|F_{\ell,\ell'}(r,r')|\leq2\pi\big(1+\norm{rw(r)}_{L^\ii}\big)\int_{-1}^1\frac{dt}{\sqrt{r^2 + r'^2 - 2 rr' t }}=\frac{4\pi\big(1+\norm{rw(r)}_{L^\ii}\big)}{\max(r,r')}$$
which is nothing else but Newton's Theorem. For the second bound \eqref{estim_F_ell_diff_sup}, we treat $w$ and the Newton interaction separately. Recall that
$$\frac{1}{\sqrt{r^2 + r'^2 - 2 rr' t }}=\sum_{m\geq0}\frac{\min(r,r')^m}{\max(r,r')^{m+1}}P_m(t).$$
Differentiating explicitely we find
$$\left|r^2\partial_r\int_{-1}^1 \frac{P_\ell(t)P_{\ell'}(t)}{\sqrt{r^2 + r'^2 - 2 rr' t }}dt\right|\leq \sum_{m\geq0}(1+m)\left|\int_{-1}^1 P_\ell P_{\ell'}P_m\right|\leq C\left(1+\ell^2+(\ell')^2\right)$$
where we have used that the integral $\int_{-1}^1 P_\ell P_{\ell'}P_m$ vanishes when $m< |\ell-\ell'|$ and when $m>\ell+\ell'$, and that $\left|\int_{-1}^1 P_\ell P_{\ell'} P_m\right|\leq 2$. For the term involving $w$, we write
\begin{align*}
&\left|r^2\partial_r \int_{-1}^1P_\ell(t)P_{\ell'}(t)w\left(\sqrt{r^2 + r'^2 - 2 rr' t }\right)\, dt\right|\\
&\qquad\qquad\qquad\qquad\qquad \leq r^2 \int_{-1}^1\left|w'\left(\sqrt{r^2 + r'^2 - 2 rr' t }\right)\right|\, dt\\
&\qquad\qquad\qquad\qquad\qquad\leq \norm{(1+r^2)^{1+\epsilon}w'(r)}_{L^\ii}\, r^2 \int_{-1}^1\frac{dt}{\left(1+r^2 + r'^2 - 2 rr' t \right)^{1+\epsilon}}.
\end{align*}
Then we explicitely compute
\begin{multline*}
r^2 \int_{-1}^1\frac{dt}{\left(1+r^2 + r'^2 - 2 rr' t \right)^{1+\epsilon}}= \frac{r}{2\epsilon\, r'} \left( \frac{1}{(1+(r-r')^2)^{\epsilon}}-\frac{1}{(1+(r+r')^2)^{\epsilon}}\right)\\
= \frac{r}{2\epsilon\, r'} \frac{(1+r^2+(r')^2)^{\epsilon}\left(\left(1+\frac{2rr'}{1+r^2+(r')^2}\right)^{\epsilon}-\left(1-\frac{2rr'}{1+r^2+(r')^2}\right)^{\epsilon}\right)}{(1+(r-r')^2)^{\epsilon}(1+(r+r')^2)^{\epsilon}}\leq C_\epsilon
\end{multline*}
which ends the proof of Lemma \ref{lem:estim_F_ell}.
\end{proof}
 
The following is a simple consequence of Lemma \ref{lem:estim_F_ell}:
\begin{corollary}
\label{cor:Newt}
Let $\rho \in L^1 (\R^3)$ be spherically symmetric function. 
Then we have
\begin{equation}
\left|V\ast\rho (x)\right|\leq \frac{\left(1+\norm{rw(r)}_{L^\ii(0,\ii)}\right)\int_{\R^3}|\rho|}{|x|}
\label{estim:rho_Newt}
\end{equation}
and
\begin{multline}
\left|\nabla_x\,|x|^2\left(V\ast\rho\right) (x)\right|\\
\leq C\left(1+\norm{rw(r)}_{L^\ii(0,\ii)}+C_\epsilon\norm{(1+r^2)^{1+\epsilon}w'(r)}_{L^\ii(0,\ii)}\right)\int_{\R^3}|\rho|.
\label{estim:rho_Newt_diff}
\end{multline}
\end{corollary}
\begin{proof}
Integrating out the angle, we find
$$V\ast\rho(r)=\int_0^{\ii}(r')^2dr'\, F_{00}(r,r')\rho(r')$$
hence the result is a simple application of Lemma \ref{lem:estim_F_ell}, with $\ell=\ell'=0$ and using that $\left|\nabla|x|^2\left(V\ast\rho\right)\right|\leq 2|x|\,\left|V\ast\rho (x)\right|+|x|^2\left|\nabla V\ast\rho (x)\right|$.
\end{proof}

\subsubsection{Pseudo-Relativistic Kinetic Energy}
We now define the positive operator $K_\ell$ acting on the space $L^2_r := L^2([0,\ii),r^2dr)$ by 
\begin{equation}
K_\ell^2:=-\frac{1}{r^2}\frac{\partial}{\partial r}r^2\frac{\partial}{\partial r} +m^2+ \frac{\ell(\ell+1)}{r^2}, 
\label{def:K_ell}
\end{equation}
i.e. $K_\ell$ is nothing but the restriction of $\sqrt{m^2-\Delta}$ to the $\ell$-th angular momentum sector. We note that, 
for all $k=-\ell,\ldots,\ell$, 
$$(m^2-\Delta )Y_\ell^k(\omega_x)u(|x|)=Y^k_\ell(\omega_x)(K_\ell^2 u)(|x|),$$
and hence 
$$\sqrt{m^2-\Delta} Y_\ell^k(\omega_x)u(|x|)=Y^k_\ell(\omega_x)(K_{\ell}\, u)(|x|).$$ 
We will need two important results involving the operator $K_\ell$. The first one is the
\begin{lemma}[Commutator Estimate]\label{lem:Stein}
Suppose $m>0$. We have, for all $f \in W^{1,\infty} (\R^3)$, 
\begin{equation}\label{Stein0}
\left\| \left[ \sqrt{m^2-\Delta} , f(x) \right] \right\|_{\cB(L^2(\R^3))} \leq C \| \nabla f \|_{\infty} \, .
\end{equation}
In particular, we have for all radial function $f\in W^{1,\ii}(\R^3)$ and all nonnegative integer $\ell$,
\begin{equation}\label{Stein1}
\norm{\big[ K_\ell, f(r) \big]}_{\cB(L^2([0,\ii),r^2dr))}  \leq C\norm{\partial_r f(r)
}_{L^\infty(0,\ii)}.
\end{equation}
\end{lemma}

A proof of \eqref{Stein0} can be found in \cite{Stein-93}; see in particular the Corollary on page~309. The statement of this corollary does not give an effective bound on the norm of the commutator. However, the corollary is based on Theorem 3, on page~294 of \cite{Stein-93}, whose proof provides the effective control we need (a remark in this sense can be found in the paragraph 3.3.5, on page 305 of \cite{Stein-93}). On the other hand, \eqref{Stein1} can be proven by applying \eqref{Stein0} in the subspace of $L^2(\R^3)$ consisting of functions of the form $Y^k_\ell(\omega_x)u(|x|)$, which is stabilized by the operator $[\sqrt{m^2-\Delta},f]$.

\medskip

The second important result that will be very useful is the following
\begin{lemma}[Estimating $K_\ell-K_{\ell'}$]\label{lem:diff_K_ell} 
There exists a constant $C$ such that one has, for all nonnegative integers $\ell,\ell'$,
\begin{equation}\label{Stein2}
\norm{\left(K_\ell - K_{\ell'}\right) r}_{\cB(L^2([0,\ii),r^2dr))}\\
 \leq C (1+\ell+\ell')|\ell-\ell'|.
\end{equation} 
\end{lemma}

The proof of Lemma \ref{lem:diff_K_ell} cannot be directly deduced from the literature,  and we provide its proof in Section \ref{app:proof_lem:Stein} below.

\subsection{Hartree-Fock Theory: Proof of Theorem \ref{thm:HF}}\label{sec:proof_HF}
\subsubsection*{Step 1. Conservation of angular momentum.}
First of all, the proof of Theorem \ref{thm:HF} makes use of the fact that (\ref{eq:HF1}) preserves both the spherical symmetry and the total angular momentum. We recall that $L=-ix\wedge\nabla$.
\begin{lemma}[Conservation of spherical symmetry and angular momentum]\label{lem:conservation_spherical_sym}
Let $\gamma_0\in \cK_{\rm HF}$ be a spherically symmetric HF state. Then the unique maximal solution $\gamma_t$ to (\ref{eq:HF2}) with $\gamma_{|t=0} = \gamma_0$ is spherically symmetric for all times $t\in[0;T)$.

If furthermore $\tr(|L|^2\gamma_0)<\ii$, then $\gamma_t$ satisfies for all times $[L,\gamma_t]=0$ and the total angular momentum is conserved: 
\[ \tr \; |L|^2  \gamma_t = \tr \; |L|^2 \gamma_0 \qquad \text{for all } t \in [0,T) . \]
\end{lemma}
Note in particular Lemma \ref{lem:conservation_spherical_sym} says that $\tr|L|^2\gamma_t<\ii$ for all $t\in[0,T)$.

The proof of this lemma can be found in \cite{HaiSch-09} (it is trivial to see that the introduction of the regular and spherically symmetric potential $w$ and the fact that $\gamma$ may have infinite rank do not affect the proof of this lemma). The conservation of the squared angular momentum as expressed by Lemma \ref{lem:conservation_spherical_sym} will be a key tool in our proof, as it was already in \cite{HaiSch-09}.

\medskip

For the rest of the proof, we introduce the non-negative self-adjoint operator 
$$M:=\sum_{i=1}^3x_i \sqrt{-\Delta+m^2} \,x_i.$$
As in \cite{FroLen-07,HaiSch-09}, the strategy of the proof consists in showing that the second derivative of the expectation $\tr M \gamma_t$ is negative, if the energy $\cE_{\rm HF} (\gamma_t) = \cE_{\rm HF} (\gamma_0)$ satisfies a certain condition. 

\subsubsection*{Step 2. Estimate on direct and exchange terms.} We will prove in this step that there exists a constant $C$ (depending on $w$) such that \begin{equation}\label{eq:step1} \frac{d}{dt} \, \tr \, M \gamma_t \leq \tr\, A \gamma_t + C \tr (\gamma_0)\; \tr (1+|L|^2)\gamma_0  \end{equation} with $A = x \cdot p + p \cdot x$. Here and henceforth, we use the notation $p = -i\nabla_x$ for the momentum operator and we use the notation $C$ to denote any constant which only depends on $w$. The explicit dependence in $w$ is very easy to derive, as it only occurs from the estimates of Lemma \ref{lem:estim_F_ell}.

\begin{remark}
{\em One can check that $\tr ( (1+|x|^4 + |p|^2) \gamma_t ) < +\infty$ for $t \in [0,T)$, provided that initially $\tr ( (1+|x|^4 + |p|^2) \gamma_0) < +\infty$ holds. This fact may be used to show that all terms below are well-defined. Also, we note that
$$
\tr (M \gamma_t) \leq C \tr ( 1+ |x|^4 + |p|^2 ) \gamma_t)< \infty
$$
and 
$$
\tr (|A| \gamma_t) \leq C \tr ( (1+|x|^4 + |p|^2) \gamma_t ) < \infty
$$
for all times $t \in [0,T)$. }
\end{remark}

\begin{remark}
{\em In order to have well-defined terms, we impose the higher regularity condition $\| \gamma_0 \|_{\mathcal{H}^{2}} < \infty$; see also Remark 1 iii) above. Note that a breakdown of the solution in $\| \cdot \|_{\cH^{s}}$ for some $s \geq 1/2$ implies breakdown of the solution in energy norm $\| \cdot \|_{\cH^{1/2}}$; this claim follows easily by adapting the arguments in \cite{Lenzmann-07}. }
\end{remark}

We start by computing 
\begin{equation}\label{eq:st1} \frac{d}{dt} \tr M \gamma_t = \tr A \gamma_t - i \tr \left[ M, V*\rho_{\gamma_t}\right] \gamma_t + i \tr \left[ M, R_{\gamma_t} \right] \gamma_t , \end{equation} 
where we recall that
$$R_\gamma(x,y)=V(|x-y|)\gamma(x,y)=-\frac{\gamma(x,y)}{|x-y|}+\gamma(x,y)w(|x-y|).$$
To prove (\ref{eq:step1}), we have to bound the last two terms on the r.h.s. of the last equation, arising, respectively, from the direct and the exchange term in the Hartree-Fock equation (\ref{eq:HF1}). We will estimate $\tr A \gamma_t$ later in Step 3.

The appropriate estimate on the direct term is given in the following
\begin{lemma}[Estimating the direct term] \label{lem:direct_term}
 We have for all spherically symmetric $\gamma\in\cK_{\rm HF}$ 
\begin{equation}
\left| \tr \left[M, \big(V\ast\rho_\gamma\big)\right] \gamma \right| \leq  C\; \big(\tr(\gamma)\big)^2
\label{eq:direct}
\end{equation}
where $C$ only depends on $w$.
\end{lemma}

For the exchange term, we have to face the problem that $\gamma_t$ can \emph{a priori} be infinite rank, hence the method of \cite{HaiSch-09} does not apply directly. However, an estimate similar to what was proved in \cite{HaiSch-09} holds true, as expressed in the following 
\begin{lemma}[Estimating the exchange term]\label{lem:exchange_term}
We have for all spherically symmetric $\gamma\in\cK_{\rm HF}$ such that $\tr|L|^2\gamma<\ii$
\begin{equation}
\left| \tr \left[M, R_\gamma\right] \gamma \right| \leq  C\; \tr(\gamma)\;\tr (1+|L|^2)\gamma 
\label{eq:exchange}
\end{equation}
where $C$ only depends on $w$.
\end{lemma}

Inserting the estimates (\ref{eq:direct}) and \eqref{eq:exchange} in the expression (\ref{eq:st1}), one gets the bound (\ref{eq:step1}). We now provide the
\begin{proof}[Proof of Lemma \ref{lem:direct_term}]
Following the strategy of \cite{FroLen-07}, we find
\begin{multline}
  i \tr \, \left[ M, V*\rho_{\gamma} \right] \gamma = i \, \tr \left[ \sqrt{|p|^2 + m^2}, |x|^2 (V * \rho_{\gamma})\right]  \gamma \\ - 2 \,\text{Re} \, \tr \, \frac{p}{\sqrt{|p|^2 + m^2}} \cdot x (V*\rho_{\gamma}) \gamma
\end{multline}
which implies that 
\[ \left| \tr \, \left[ M, (V*\rho_{\gamma}) \right] \gamma_t \right| \leq \left( \left\| \left[ \sqrt{|p|^2 + m^2}, |x|^2 (V*\rho_{\gamma}) \right] \right\| + 2 \left\| |x| (V*\rho_{\gamma})\right\|_{\infty} \right) \tr(\gamma)\, . \]
Applying Lemma \ref{lem:Stein} to the first term in the parenthesis, we obtain
\[  \left| \tr \, \left[ M, (V*\rho_{\gamma})\right] \gamma \right| \leq \left( \| \nabla_x\,|x|^2 V*\rho_{\gamma} \|_{\infty} + \| |x| (V*\rho_{\gamma}) \|_{\infty} \right) \tr(\gamma) \, . \]
Using the spherical symmetry of $\rho_{\gamma}$, \eqref{eq:direct} is then a consequence of Corollary \ref{cor:Newt}.
\end{proof}

It rests to write the 
\begin{proof}[Proof of Lemma \ref{lem:exchange_term}]
 First we write as before
\begin{multline} \label{exchange}
i\tr \left[M, R_\gamma\right] \gamma =
i\tr\left(\sqrt{|p|^2 +m^2} |x|^2 R_\gamma \gamma\right)\\ - i\tr\left(R_\gamma |x|^2 \sqrt{|p|^2+m^2} \gamma\right) 
+ 2 \text{Re} \tr \left(R_\gamma x\cdot \frac p{\sqrt{|p|^2+m^2}}
\gamma\right).
\end{multline}
We claim that
\begin{multline}
 \left|\tr\left(\sqrt{|p|^2 +
1} |x|^2 R_\gamma \gamma\right)- \tr\left(R_\gamma |x|^2 \sqrt{|p|^2+m^2} \gamma\right)\right| \leq C\tr (\gamma)\; \tr (1 + |L|^2)\gamma,
\label{estim_exchange_1}
\end{multline}
and that
\begin{equation}
 \left|\tr \left(R_\gamma x\cdot \frac p{\sqrt{|p|^2+m^2}}\right)\gamma\right|\leq C (\tr\gamma)^2.
\label{estim_exchange_2}
\end{equation}

To prove (\ref{estim_exchange_1}) and (\ref{estim_exchange_2}), we observe that, because of the spherical symmetry of $\gamma$, we can expand the kernel of $\gamma$ as 
\begin{equation}
\gamma(x,y)=\sum_{\ell\geq0}g_\ell(|x|,|y|)(2\ell+1)P_\ell(\omega_x\cdot\omega_y)
\label{eq:form_gamma}
\end{equation}
where $\omega_x=x/|x|$, $\omega_y=y/|y|$ and $P_\ell$ is the $\ell$-th Legendre polynomial whose formula was recalled in \eqref{def:Legendre}. 

We have $$ \sqrt{|p_x|^2 + 1} \, g_\ell(|x|,|y|) \, P_\ell(\omega_x\cdot\omega_y) = P_\ell(\omega_x\cdot\omega_y) \, (K_\ell)_r \, g_\ell(|x|,|y|) $$
where $K_\ell$ was defined in \eqref{def:K_ell}. The subscript $x$ in $p_x^2$ and the subscript $r=|x|$ in $(K_\ell)_r$ indicate that these operators act on the $x$, respectively, $r=|x|$ variable. 
It follows that
\begin{equation}\label{eq:ex1}
\begin{split}
&\tr\left(\sqrt{|p|^2+m^2}  |x|^2 R_\gamma \gamma\right)\\
&\qquad\qquad=\int_{\R^6}  \gamma(y,x) \sqrt{|p_x|^2 + m^2} \left(|x|^2 R_\gamma(x,y)\right) dy dx\\
&\qquad\qquad= \sum_{\ell,\ell'}(2\ell +1)(2\ell'+1) \int_{\R^6}  g_{\ell'}
(|y|,|x|)P_{\ell'}(\omega_x\cdot\omega_y)\times\\
&\qquad\qquad\qquad\qquad\times\sqrt{|p_x|^2 +m^21}\left( |x|^2
g_\ell(|x|,|y|) P_\ell(\omega_x\cdot\omega_y) V (x-y) \right) dy \,dx\\
&\qquad\qquad=\sum_{\ell,\ell'} (2\ell + 1) (2\ell'+1)\left\langle g_{\ell'}(r,r'),
(K_{\ell'})_r  r^2 F_{\ell,\ell'}(r,r') g_\ell(r,r') \right\rangle_{L^2_r\otimes L^2_{r}},
\end{split} \end{equation}
where $F_{\ell,\ell'}$ was defined before in \eqref{def:F_ell}.
Similarly to (\ref{eq:ex1}) we obtain
\begin{equation}
\tr\left(R_\gamma |x|^2 \sqrt{|p|^2+m^2} \gamma\right) =
\sum_{\ell,\ell'}(2\ell + 1) (2\ell'+1)\left\langle g_{\ell'},
F_{\ell,\ell'} r^2 \left(K_{\ell}\right)_r g_\ell\right\rangle_{L^2_r\otimes L^2_{r}}.
\end{equation}
Therefore we can rewrite the left hand side of \eqref{estim_exchange_1} in
the form
\begin{align}
& \tr\left(\sqrt{|p|^2 +m^2} |x|^2 R_\gamma \gamma\right) - \tr\left(R_\gamma |x|^2 \sqrt{|p|^2+m^2}\, \gamma\right)\nonumber\\
&\quad=\sum_{\ell,\ell'}(2\ell + 1)(2\ell' + 1)\left\langle g_{\ell'},
 \left( \left(K_{\ell'}\right)_r r^2 F_{\ell,\ell'} - r^2
F_{\ell,\ell'} \left(K_\ell\right)_r\right)g_\ell \right \rangle_{L^2_r\otimes L^2_r}\nonumber\\ 
&\quad= \sum_{\ell,\ell'} (2\ell + 1)(2\ell' + 1)\bigg(\left\langle
g_{\ell'} , (K_{\ell} - K_{\ell'})_rr^2 F_{\ell,\ell'}g_\ell\right
\rangle_{L^2_r\otimes L^2_r}\nonumber\\
&\qquad\qquad\qquad\qquad\qquad + \left\langle g_\ell , \left[  \left(K_{\ell}\right)_r, r^2
F_{\ell,\ell'}\right] g_{\ell'}\right\rangle_{L^2_r\otimes L^2_r}\bigg).\label{eq:decomp_exchange}
\end{align}

Now using both Lemma \ref{lem:estim_F_ell} and Lemma \ref{lem:diff_K_ell}, we get
$$ \norm{ [K_\ell- K_{\ell'}] r^2 F_{\ell,\ell'}} \leq
C(1+\ell+\ell')|\ell-\ell'|.$$ This allows us to control the first term on the right side of 
(\ref{eq:decomp_exchange}). To control the second term, we remark that, from Lemma \ref{lem:Stein}, we have
$$ \norm{[K_\ell,|x|^2 F_{\ell,\ell'}]} \leq \norm{\partial_r r^2 F_{\ell,\ell'}}_{L^\infty(0,\infty)^2}\leq C.$$ 
Inserting these estimates in \eqref{eq:decomp_exchange} and noting that
$$\norm{g_\ell}_{L^2_r\otimes L^2_r}=\norm{g_\ell}_{\gS_2(L^2_r)}\leq \norm{g_\ell}_{\gS_1(L^2_r)}=\tr_{L^2_r} g_\ell$$
yields
\[ \begin{split} 
 \Big| \, \tr\Big(\sqrt{|p|^2 +
1} |x|^2 R_\gamma &\gamma\Big)- \tr\left(R_\gamma |x|^2 \sqrt{|p|^2+m^2} \gamma\right)\Big| \\ \leq \; &C \sum_{\ell, \ell'} (2\ell+1) (2\ell'+1) \left( 1 + (1+\ell + \ell') |\ell- \ell'| \right) \, \tr_{L^2_r} \, g_{\ell} \, \tr_{L^2_r} g_{\ell'} \\ 
\leq \; &C \sum_{\ell, \ell'} (2\ell+1) (2\ell'+1) \left( 1 + \ell^2 + (\ell')^2 \right) \, \tr_{L^2_r} \, g_{\ell} \, \tr_{L^2_r} g_{\ell'} \\ \leq \; &C\; \tr\gamma\;\tr(1+|L|^2)\gamma
\end{split}\]
and thus implies \eqref{estim_exchange_1}.

It remains to prove \eqref{estim_exchange_2}. Since $x\cdot p
= r \partial_r$, with $|x| = r$, we obtain
\[ \begin{split}
\Big| \tr\Big( R_\gamma x &\cdot \frac p {\sqrt{|p|^2+m^2} } \gamma \Big) \Big| \\ \leq \; &
\sum_{\ell,\ell'}(2\ell + 1)(2\ell' + 1) \left| \left\langle g_\ell ,
F_{\ell,\ell'} r
\partial_r \frac 1 {K_{\ell'}} g_{\ell'}
\right \rangle_{L^2_r\otimes L^2_r} \right|
\\ \leq \; &\sum_{\ell,\ell'}(2\ell + 1)(2\ell' + 1)\norm{g_\ell}_{\gS_2(L^2_r)}  \norm{g_{\ell'}}_{\gS_2(L^2_r)} \norm{
r F_{\ell,\ell'}}_{L^\ii}  \norm{ \partial_r  (K_{\ell'})^{-1}}_{\cB(L^2_r)} \\  \leq \; &
(\tr\gamma)^2 \end{split} \]
as was claimed. Note that we have used that $(\partial_r)^*\partial_r=-r^{-2}\partial_r r^2\partial_r\leq K_{\ell'}^2$ for all $\ell'\geq0$. Hence $\norm{\partial_r(K_{\ell'})^{-1}}\leq1$ as an operator acting on $L^2_r$.
This ends the proof of Lemma \ref{lem:exchange_term}.
\end{proof}

\subsubsection*{Step 3. Estimating $\tr(A\gamma_t)$.} We recall that $A = x \cdot p + p \cdot x$. We will show in this step that 
\begin{equation}\label{eq:step2} \frac{d}{dt} \, \tr \, A \gamma_t \leq 2 \, \cE_{\rm HF} (\gamma_0) +  \kappa \, (\tr\gamma_0)^2 \, \sup_{r\geq0} \left| w(r)+rw'(r) \right|_-. \end{equation}

In order to prove (\ref{eq:step2}), we compute as before \[ \begin{split} \frac{d}{dt} \, \tr \, A \gamma_t = \; & -i \, \tr A [\sqrt{-\Delta + m^2} +\kappa (V * \rho_{\gamma_t}) - \kappa R_{\gamma_t} , \gamma_t ] \\ = \; &-i \tr \, [A , \sqrt{-\Delta + m^2} +\kappa (V * \rho_{\gamma_t}) - \kappa R_{\gamma_t}] \gamma_t \\ = \; & 2 \, \tr \, \frac{p^2}{\sqrt{p^2 + m^2}} \gamma_t \\ &- \kappa \int dx\, dy\, (x-y) \cdot \nabla V (x-y) \left( \gamma_t (x,x) \gamma_t (y,y) - |\gamma_t (x,y)|^2 \right) \,. \end{split} \]
With $V(x) = -|x|^{-1} + w(x)$, we obtain \[ x \cdot \nabla V (x) = \frac{1}{|x|} + x \cdot \nabla w (x) = - V(x) + x \cdot \nabla w (x) + w (x) \, .\]
Since moreover $p^2/ \sqrt{p^2 + m^2} \leq \sqrt{p^2 + m^2}$, we find
\[ \begin{split} \frac{d}{dt} \, \tr \,& A \gamma_t  \\ \leq \; & 2 \, \cE_{\rm HF} (\gamma_t) \\ &- \kappa \int dx dy \left( (x-y) \cdot \nabla w (x-y) + w(x-y) \right) \, \left (\gamma_t (x,x) \gamma_t (y,y) - |\gamma_t (x,y)|^2\right) \\ \leq \; & 2 \, \cE_{\rm HF} (\gamma_t) + \kappa \, (\tr\gamma_t)^2 \, \sup_{r\geq0} |w(r)+rw'(r)|_- \, , \end{split} \] where $|f|_-\geq0$ denotes the negative  part of $f$. Here we used the fact that $\gamma_t (x,x)\gamma_t (y,y) \geq |\gamma_t (x,y)|^2$ since $\gamma \geq 0$. Equation (\ref{eq:step2}) thus follows because $\cE_{\rm HF} (\gamma_t) = \cE_{\rm HF} (\gamma_0)$ and $\tr(\gamma_t) = \tr (\gamma_0)$. 

\subsubsection*{Step 4. Conclusion of the proof of Theorem \ref{thm:HF}.} It follows from Step~2 and Step~3, that 
\begin{multline}\label{eq:concl} 0 \leq\tr \, M \gamma_t \leq \;  t^2 \left( \cE_{\rm HF} (\gamma_0)  + \frac{\kappa}{2} \, (\tr\gamma_0)^2 \sup_{r\geq0} |w(r)+rw'(r)|_-  \right) \\ + t \bigg( \tr (A \gamma_0) + C\; \tr (\gamma_0) \; \tr(1+|L|^2)\gamma_0 \bigg) + \tr M \gamma_0  
\end{multline}
for all $t< T$ (recall that $T$ is the time of existence of the maximal local solution) and where $C$ is a constant which only depends on $w$. {F}rom the assumptions on the initial datum $\gamma_0$, we find \[ \tr(M \gamma_0) \leq C\tr (1+|x|^4 + |p|^2) \gamma_0 < \infty \] and \[ |\tr (A\gamma_0)| = |\tr (x\cdot p + p \cdot x) \gamma_0| \leq \tr (|x|^2+|p|^2) \gamma_0  < \infty\, . \] 
Hence it is clear that, if $\cE_{\rm HF} (\gamma_0) + (\kappa/2) (\tr\gamma_0)^2 \, \sup_{r\geq0} |w(r)+rw'(r)|_- < 0$, our estimate (\ref{eq:concl}) contradicts the positivity of the operator $M$, for $t$ large enough. Therefore $T < \infty$, and, from the blowup alternative (see Theorem \ref{thm:LWP}), it follows that $\| \gamma_t \|_{\cX_{\rm HF}} \to \infty$, as $t \to T^-$.
\qed

\subsection{Hartree-Fock-Bogoliubov Theory: Proof of Theorem \ref{thm:HFB_cut_off}}\label{sec:proof_HFB}
The proof follows the same lines as the one of the HF case and not all the details will be provided. We take as before $M=\sum_{i=1}^3x_i\,\sqrt{m^2-\Delta}\,x_i$ and note that $M$ commutes with $L$, hence it also commutes with our cutoff projector $P_\Lambda$. The same calculation as in \eqref{eq:st1} yields
\begin{multline}
\frac{d}{dt}\tr(M\gamma_t)=-i\,\tr\big([M,P_\Lambda H_{\gamma_t}P_\Lambda ]\gamma_t\big)\\
+\frac{i}2\kappa\pscal{\alpha_t, \left[M_x + M_y, P_\Lambda\otimes P_\Lambda\left(\frac 1{|x-y|}-w(|x-y|)\right)P_\Lambda\otimes P_\Lambda\right]\alpha_t}_{L^2(\R^6)}.
\label{comput_M_2}
\end{multline}
However, since $[M,P_\Lambda]=0$, and $\gamma_t$ and $\alpha_t$ are easily seen to satisfy \eqref{HFB_mom_cut_off} for all times, we may simply rewrite this as
\begin{multline}
\frac{d}{dt}\tr(M\gamma_t)=\tr\,A\gamma_t-i\,\tr\left[M,V\ast\rho_{\gamma_t}\right]\gamma_t+i\,\tr\,\left[M,R_{\gamma_t}\right]\gamma_t\\
-\frac{i}2\kappa\pscal{\alpha_t, \left[M_x + M_y, V(x-y)\right]\alpha_t}_{L^2(\R^6)}.
\label{comput_M_3}
\end{multline}
With the help of Lemma \ref{lem:direct_term} and \ref{lem:exchange_term}, we deduce 
$$\left|\tr\left[M,V\ast\rho_{\gamma_t}\right]\gamma_t\right|\leq C\big(\tr\,\gamma_t\big)^2$$
and
$$\left|\tr\,\left[M,R_{\gamma_t}\right]\gamma_t\right|\leq C\,\tr\,\gamma_t\;\tr(1+|L|^2)\gamma_t.$$
Now, we  need an estimate on the pairing term, i.\,e.~the last term of \eqref{comput_M_3}. This is the goal of the
\begin{lemma}[Estimating the pairing term]\label{lem:estim_alpha}
Let $(\gamma,\alpha)\in\cK_{\rm HFB}$ be a radially symmetric HFB state such that $\tr(|L|^{6+\epsilon}\gamma)<\ii$ for some $\epsilon>0$. Then one has 
\begin{multline}
\left|\pscal{\alpha, \left[M_x + M_y, V(x-y)\right]\alpha}_{L^2(\R^6)}\right|\\
\leq C\big(\tr(1+|L|^3)\gamma\big)^{1/2}\big(\tr(1+|L|^{6+\epsilon})\gamma\big)^{1/2}\left(\sum_{\ell}\frac{1}{1+\ell^{1+\epsilon}}\right)^{1/2}.
\label{estim_alpha_M_1}
\end{multline}
where $C$ only depends on $w$.
\end{lemma}

Assuming that Lemma \ref{lem:estim_alpha} holds true, we can conclude the proof of Theorem \ref{thm:HFB_cut_off}. Combining estimates, we arrive at the bound
\begin{align*}
&\frac{d}{dt}\tr(M\gamma_t)\\
&\qquad\leq \tr\,A\gamma_t +C\,\tr(\gamma_t)\,\tr(1+|L|^2)\gamma_t+C\tr(1+|L|^3)\gamma_t\big)^{1/2}\big(\tr(1+|L|^{6+\epsilon})\gamma_t\big)^{1/2}\\
 &\qquad\leq \tr\,A\gamma_t +C\left(1+\Lambda^{\frac{9+\epsilon}{2}}\right)\,\big(\tr\,\gamma_t\big)^2
\end{align*}
where we have used that $P_\Lambda\gamma=\gamma$.

The next step is to note that $[P_\Lambda,A]=0$, hence we have like in the HF case
\begin{multline} 
\frac{d}{dt} \, \tr \, A \gamma_t \leq 2 \, \cE_{\rm HFB} (\gamma_t,\alpha_t) \\
- \kappa \int dx\, dy\, \big((x-y) \cdot \nabla w (x-y)+w(x-y)\big) \left( \gamma_t (x,x) \gamma_t (y,y) - |\gamma_t (x,y)|^2+|\alpha_t(x,y)|^2 \right).
\end{multline}
We may use as before that $\gamma_t (x,x) \gamma_t (y,y) - |\gamma_t (x,y)|^2\geq0$ and that $\norm{\alpha_t}_{L^2(\R^6)}^2\leq \tr(\gamma_t)$ to infer
$$
\frac{d}{dt} \, \tr \, A \gamma_t \leq 2 \, \cE_{\rm HFB} (\gamma_t,\alpha_t) \\
+ \kappa \sup_{r\geq0}|w(r)+rw'(r)|_-\left(\big(\tr\gamma_t\big)^2+\tr\gamma_t\right).
$$
The end of the proof is then the same as in the HF case, using that $\tr\gamma_t=\tr\gamma_0$ for all times.

\medskip

It now rests to give the
\begin{proof}[Proof of Lemma \ref{lem:estim_alpha}]
We write as usual
\begin{equation}
\gamma(x,y)=\sum_{\ell\geq0}g_\ell(|x|,|y|)(2\ell+1)P_\ell(\omega_x\cdot\omega_y)
\label{eq:form_gamma2}
\end{equation}
where $g_\ell(r,r')$ is the kernel of a self-adjoint  operator $0\leq g_\ell\leq 1$ acting on radial functions in $L^2(\R^3;\C)$, and
\begin{equation}
\alpha(x,y)=\sum_{\ell\geq0}a_\ell(|x|,|y|)(2\ell+1)P_\ell(\omega_x\cdot\omega_y).
\label{eq:form_alpha2}
\end{equation}
This time the operator $a_\ell$ is antisymmetric, $a_\ell(r,r')=-a_\ell(r',r)$, i.e. its kernel $a_\ell(|x|,|y|)$ can be interpreted as a spherically symmetric, fermionic two-body wavefunction.

We also recall \cite{BacLieSol-94} that the condition $0\leq\Gamma\leq1$ gives $\alpha \alpha^*\leq \gamma(1-\gamma)$. The same inequality must hold on each angular momentum sector, hence we obtain the following inequality for operators acting on $L^2_r$: $a_\ell a_\ell^*\leq g_\ell(1-g_\ell)$.


Now, arguing as for the exchange term, we calculate:
\begin{multline}\label{pairing-one}
i\left \langle \alpha , \left[M_x + M_y, V(x-y)\right] \alpha
\right\rangle_{L^2(\R^6)}\\ = 2i \left \langle \alpha,\left[ \sqrt{|p_x|^2 +
1} , {|x|^2}V(x-y) \right] \alpha \right \rangle_{L^2(\R^6)} \\ - 4\Re
\left\langle \alpha,\frac{p_x}{\sqrt{|p_x|^2+1}} \cdot x V(x-y)
\alpha\right\rangle_{L^2(\R^6)}.
\end{multline}
For the second term on the right hand side of
\eqref{pairing-one} we obtain
\begin{multline}\label{pairA}
\left \langle \alpha , \frac {\partial_r }{\left(K_\ell\right)_r } r V(x-y)
\alpha\right \rangle_{L^2(\R^6)} \\ = \sum_{\ell,\ell'} (2\ell + 1)(2\ell'+1)
\left\langle a_\ell , \frac{\partial_r }{\left(K_\ell\right)_r} r
F_{\ell,\ell'}a_{\ell'}\right\rangle_{L^2(\R^6)},
\end{multline}
where $F_{\ell,\ell'}$ was defined before in \eqref{def:F_ell}. This yields to 
$$ |\eqref{pairA}| \leq
  \sum_{\ell,\ell'} (2\ell +
1)(2\ell'+1)\norm{a_\ell}_{\gS_2(L^2_r)}
\norm{a_\ell}_{\gS_2(L^2_r)}\norm{rF_{\ell,\ell'}}_{L^\infty(0,\infty)^2}.
$$
By our assumptions, we have 
$a_\ell a_\ell^* \leq g_\ell(1-g_\ell)$, hence
$$\norm{a_\ell}_{\gS_2(L^2_r)} \leq
(\tr_{L^2_r} g_\ell)^{1/2}.$$ 
Together with Lemma \ref{lem:estim_F_ell},
we can further bound \eqref{pairA} using the Cauchy-Schwarz inequality as
$$|\eqref{pairA}|\leq  \left( \sum_\ell (2\ell + 1) (\tr_{L^2_r\otimes\C^2} g_\ell)^{1/2}\right)^2
\leq  C\tr \left(1+|L|^{4}\right)\gamma.$$
Note the previous bound can be improved to $\leq C\tr \left(1+|L|^{2+\epsilon}\right)\gamma$ for $\epsilon>0$.
Similarly as in the proof of Lemma \ref{lem:exchange_term}, the first term in \eqref{pairing-one} can be expressed as
\begin{multline} \sum_{\ell,\ell'} (2\ell + 1)(2\ell'+1)\left\langle a_{\ell'} ,\big( \left(K_{\ell'}\right)_r r^2
F_{\ell',\ell}  -  r^2 F_{\ell',\ell} \left(K_\ell\right)_r\big)a_{\ell'}\right
\rangle_{L^2_r\otimes L^2_r}\\ =  \sum_{\ell,\ell'} (2\ell + 1)(2\ell'+1)
\left \langle a_{\ell'}, \big( K_{\ell'} -K_{\ell}\big)_r r^2
F_{\ell',\ell}a_\ell \right\rangle_{L^2_r\otimes L^2_r} \\ + \sum_{\ell,\ell'}
(2\ell + 1)(2\ell'+1)\left\langle a_{\ell'},[ \left(K_\ell\right)_r, r^2
F_{\ell',\ell} ]a_\ell \right \rangle_{L^2_r\otimes L^2_r}.
\end{multline}
Using Lemmas \ref{lem:estim_F_ell}, \ref{lem:Stein} and \ref{lem:diff_K_ell} like for the exchange term, we obtain
\begin{align*}
 &\left|\sum_{\ell,\ell'} (2\ell + 1)(2\ell'+1)\left\langle a_{\ell'} ,\big( K_{\ell'} |x|^2
F_{\ell',\ell}  -  |x|^2 F_{\ell',\ell} K_\ell\big)a_{\ell'}\right
\rangle_{L^2_r\otimes L^2_r}\right|\\
&\qquad\qquad\leq C\sum_{\ell,\ell'} (2\ell + 1)(2\ell'+1)(1+(\ell+\ell')|\ell-\ell'|)\norm{a_\ell}_{\gS_2}\norm{a_{\ell'}}_{\gS_2}\\
&\qquad\qquad\leq C\left(\sum_{\ell} (2\ell + 1)(1+\ell^2)\tr_{L^2_r}(g_\ell)^{1/2}\right)\left(\sum_{\ell'}(2\ell'+1)\tr_{L^2_r}(g_{\ell'})^{1/2}\right)\\
&\qquad\qquad\leq C\big(\tr(1+|L|^3)\gamma\big)^{1/2}\big(\tr(1+|L|^{6+\epsilon})\gamma\big)^{1/2}\left(\sum_{\ell}\frac{1}{1+\ell^{1+\epsilon}}\right)^{1/2}.
\end{align*}
This ends the proof of Lemma \ref{lem:estim_alpha}.
\end{proof}

\begin{appendix}
\section{Proof of Lemma \ref{lem:diff_K_ell}}\label{app:proof_lem:Stein}
We have, using the well-known \cite{LieLos-01} integral formula for $\sqrt{-\Delta + m^2}-m$ and assuming that $u,v$ are smooth enough,
\begin{align*}
&\pscal{\left(K_{\ell}-m\right)\,u,v}_{L^2([0,\ii),r^2dr)}\\
&=\pscal{(\sqrt{m^2-\Delta}-m)Y_\ell^k(\omega_x)u(|x|),Y_\ell^k(\omega_x)v(|x|)}_{L^2(\R^3)}\\
&=\frac{1}{4\pi^2}\iint_{\R^6}\frac{\cK_2(|x-y|)}{|x-y|^2}\big(\overline{Y_\ell^k(\omega_x)}\overline{u(|x|)}-\overline{Y_\ell^k(\omega_y)}\overline{u(|y|)}\big)\times\\
&\qquad\qquad\qquad\qquad\times\big(Y_\ell^{k}(\omega_x){v(|x|)}-Y_\ell^{k}(\omega_y){v(|y|)}\big)dx\,dy
\end{align*}
where $\cK_2$ is a Bessel function.
Averaging over $k$ and using that
$$\frac{1}{2\ell+1}\sum_{k=-\ell}^\ell \overline{Y_\ell^{k}(\omega)}Y_\ell^{k}(\omega')=P_\ell(\omega\cdot\omega')$$
where $P_\ell$ is the $\ell$-th Legendre polynomial, we get
\begin{multline}
\pscal{\left(K_{\ell}-K_{\ell'}\right)\,r\,u,v}_{L^2([0,\ii),r^2dr)}\\
=-\frac{1}{2\pi^2}\iint_{\R^6}\frac{\cK_2(|x-y|)}{|x-y|^2}|x|\overline{u(|x|)}v(|y|)\left(P_\ell(\omega_x\cdot\omega_y)-P_{\ell'}(\omega_x\cdot\omega_y)\right)dx\,dy.
\end{multline}
Note that we have used that for all $\ell$, $P_\ell(\omega_x\cdot\omega_x)=P_\ell(1)=1$.
Using that $|x|^2\cK_2(x)$ is bounded, this yields
\begin{multline}
\left|\pscal{\left(K_{\ell}-K_{\ell'}\right)\,r\,u,v}_{L^2([0,\ii),r^2dr)}\right|\\
\leq\frac{1}{2\pi^2}\iint_{\R^6}\frac{|x|\,|u(|x|)|\, |v(|y|)|\left|P_\ell(\omega_x\cdot\omega_y)-P_{\ell'}(\omega_x\cdot\omega_y)\right|}{|x-y|^4}dx\,dy
\end{multline}
Passing first to spherical and then to polar coordinates gives
\begin{align*}
&\left|\pscal{\left(K_{\ell}-K_{\ell'}\right)\,r\,u,v}_{L^2([0,\ii),r^2dr)}\right|\\
&\leq 4\int_{-1}^1d\alpha\int_{0}^\ii r^2dr\int_{0}^\ii s^2ds\frac{r|u(r)|\,|v(s)|\left|P_{\ell}(\alpha)-P_{\ell'}(\alpha)\right|}{\left(r^2+s^2-2rs\alpha\right)^2}\\
&=4\int_{-1}^1d\alpha\int_{0}^\ii t^2dt\int_{0}^{\pi/2}d\theta \cos^3\theta\sin^2\theta\frac{|u(t\cos\theta)|\, |v(t\sin\theta)|\,\left|P_{\ell}(\alpha)-P_{\ell'}(\alpha)\right|}{\left(1-2\sin\theta\cos\theta\alpha\right)^2}.
\end{align*}
Using the Cauchy-Schwarz inequality for the $t$ integration we get
\begin{multline*}
\left|\pscal{\left(K_{\ell}-K_{\ell'}\right)\,r\,u,v}_{L^2([0,\ii),r^2dr)}\right|\\
\leq 4\int_{-1}^1d\alpha\int_{0}^{\pi/2}d\theta \frac{\cos^{3/2}\theta\sin^{1/2}\theta\left|P_\ell(\alpha)-P_{\ell'}(\alpha)\right|}{\left(1-2\sin\theta\cos\theta\alpha\right)^2}\norm{u}_{L^2_r}\norm{v}_{L^2_r}.
\end{multline*}
Let us recall that $P_\ell(1)=1$ for all $\ell$ and introduce
$$C_{\ell,\ell'}:=\sup_{\alpha\in[-1,1]}\frac{|P_\ell(\alpha)-P_{\ell'}(\alpha)|}{1-\alpha}$$
such that
$$\norm{\big(K_{\ell}-K_{\ell'}\big)r}_{\cB(L^2([0,\ii),r^2dr))}\leq 4C_{\ell,\ell'}\int_{-1}^1d\alpha\int_{0}^{\pi/2}d\theta \frac{\cos^{3/2}\theta\sin^{1/2}\theta(1-\alpha)}{\left(1-2\sin\theta\cos\theta\alpha\right)^2}$$
which is finite as seen by computing
$$\int_{-1}^1d\alpha \frac{1-\alpha}{\left(1-u\alpha\right)^2}=\frac{(1+u)\log\left(\frac{1+u}{1-u}\right)-2u}{u^2(1+u)}\leq2-\log(1-u).$$
Hence it remains to show that
\begin{equation}
C_{\ell,\ell'}\leq C(1+\ell+\ell')|\ell-\ell'|.
\label{estim_C_ell_ell2} 
\end{equation}
The Legendre polynomials satisfy the relation
$$(1-x^2)P_n'(x)=-nxP_n(x)+nP_{n-1}(x)$$
which may be written
$$\frac{P_n(x)-P_{n-1}(x)}{1-x}=P_n(x)-(1+x)\frac{P_n'(x)}{n}.$$
Assuming for convenience $\ell'>\ell$ and summing over $n=\ell+1,\ldots,\ell'$, we obtain
$$\frac{P_{\ell'}(x)-P_{\ell}(x)}{1-x}=\sum_{n=\ell+1}^{\ell'}\left(P_n(x)-(1+x)\frac{P_n'(x)}{n}\right).$$
We have the formula
$$P_n'(x)=\frac{n(n+1)}{1-x^2}\int_x^1P_n(t)\,dt$$
which shows that $(1+x)|P_n'(x)|\leq n(n+1)$. Hence we obtain the bound
$$C_{\ell,\ell'}\leq \sum_{n=\ell+1}^{\ell'}(2+n),$$
which completes the proof of Lemma \ref{lem:Stein}. \qed

\end{appendix}

\bibliographystyle{siam}
\bibliography{biblio.bib}
\end{document}